\documentclass[11pt]{article}
\usepackage[letterpaper,margin=0.8in,headheight=14pt]{geometry}
\usepackage[T1]{fontenc}
\usepackage{newtxtext}
\usepackage{mathtools,amsthm}
\usepackage{newtxmath}
\usepackage{microtype,enumitem,fancyhdr}
\usepackage[colorlinks=true,linkcolor=blue,citecolor=blue,urlcolor=blue,hypertexnames=false]{hyperref}
\hypersetup{pdftitle={Polynomial Hessian equations},pdfauthor={Research notes}}
\newtheorem{theorem}{Theorem}[section]
\newtheorem{proposition}[theorem]{Proposition}
\newtheorem{lemma}[theorem]{Lemma}
\newtheorem{corollary}[theorem]{Corollary}
\theoremstyle{definition}
\newtheorem{definition}[theorem]{Definition}
\newtheorem{remark}[theorem]{Remark}

\theoremstyle{definition}
\newtheorem{assumption}{Assumption}[section]

\newcommand{\R}{\mathbb R}
\newcommand{\C}{\mathbb C}

\newcommand{\one}{\mathbf1}
\newcommand{\ddbar}{\sqrt{-1}\,\partial\bar\partial}
\newcommand{\Pol}{\operatorname{Pol}}
\newcommand{\tr}{\operatorname{tr}}
\newcommand{\Herm}{\operatorname{Herm}}
\newcommand{\SH}{\mathcal{SH}}

\newcommand{\E}{\mathcal E}

\newcommand{\CC}{\mathcal C}

\newcommand{\FF}{\mathcal F}
\newcommand{\s}{\mathbf s}
\newcommand{\D}{D_{\one}}
\newcommand{\frp}{\mathfrak p}
\newcommand{\K}{\mathcal K_d}

\newcommand{\Ga}{\mathrm{G}}
\DeclareMathOperator{\osc}{osc}

\DeclareMathOperator{\conv}{conv}

\setlist[enumerate,1]{label=(\roman*),leftmargin=2em}
\title{Solvability of complex Hessian equations corresponding to right-Noetherian polynomials}
\author{Junbang Liu, Yulun Xu}
\date{October 2026}

\begin{document}
\maketitle

\begin{abstract}
On compact Hermitian manifolds, we establish a framework using radial normalization for the solvability of complex Hessian equation in form type, corresponding to right-Noetherian polynomials whose coefficients may not be constant functions. The solvability of a radial normalized equation is characterized by the existence of a $\CC$-subsolution at radial slope, which extends the results of Guo-Song \cite{GS24} to this setting. 
\end{abstract}

\tableofcontents 
\section{Introduction}
Let $(X,\chi)$ be a compact connected Hermitian manifold of complex
dimension $n$, and let $\omega_0$ be a smooth real
$(1,1)$-form. For $\varphi\in C^\infty(X,\R)$, write
\[
\omega_\varphi=\omega_0+\ddbar\varphi,
\qquad
\Lambda_\varphi=\lambda_\chi(\omega_\varphi)
=\lambda\bigl[\chi^{i\bar l}(\omega_\varphi)_{j\bar l}\bigr].
\]
Let
\begin{equation}\label{e polynomial}
p_z(s)=s^d-\sum_{k=0}^{d-1}a_k(z)s^k,
\qquad 2\le d\le n,
\end{equation}
be a monic right-Noetherian polynomial for each $z\in X$, where
$a_k\in C^\infty(X,\R)$. 
We use the normalized ordinary polarization
\[
S_k(\lambda):=\frac{\sigma_k(\lambda)}{\binom nk},
\qquad
P_z(\lambda):=\Pol_n(p_z)(\lambda)
=S_d(\lambda)-\sum_{k=0}^{d-1}a_k(z)S_k(\lambda).
\]
We consider the corresponding complex
Hessian equation in form type:
\begin{equation}\label{e main eq}
\omega_\varphi^d\wedge\chi^{n-d}
-\sum_{k=0}^{d-1}a_k(z)\omega_\varphi^k\wedge\chi^{n-k}=0,
\end{equation}
which is equivalent to 
\begin{equation*}
    P_z(\lambda[\chi^{i \bar{j}} (\omega_{\varphi})_{l \bar{j}}])=0.
\end{equation*}

This includes the following equations:
\begin{enumerate}
\item The complex $k$-Hessian equation,
\begin{equation}\label{e com k hes}
\omega_\varphi^k\wedge\chi^{n-k}=e^{G(z)+C}\chi^n,
\qquad \Lambda_\varphi\in\Gamma_k.
\end{equation}
\item The complex Hessian quotient equation,
\begin{equation}\label{e com quo}
\omega_\varphi^k\wedge\chi^{n-k}
=e^{G(z)+C}\omega_\varphi^l\wedge\chi^{n-l},
\qquad \Lambda_\varphi\in\Gamma_k,
\qquad 0\le l<k\le n.
\end{equation}
\end{enumerate}
Here $G$ is prescribed, $C\in\R$ is unknown, and
\[
\Gamma_k=\{\lambda\in\R^n:\sigma_j(\lambda)>0,
\ 1\le j\le k\},\qquad \Gamma_n=(0,\infty)^n,
\]
where $\sigma_j$ denotes the $j$-th elementary symmetric polynomial; see \cite[Section 2]{CNS85} for the associated cones.

The analytic theory of Hessian equations rests on the interplay between
admissible cones, ellipticity, and concavity. The work of Caffarelli,
Nirenberg and Spruck \cite{CNS85} provides a fundamental framework for
equations depending on the eigenvalues of a Hessian. In complex geometry,
Yau's solution of the Calabi conjecture \cite{Yau78} established the
smooth solvability of the complex Monge--Amp\`ere equation on compact
K\"ahler manifolds. For the intermediate complex $k$-Hessian equations,
Hou, Ma and Wu obtained a second-order estimate in terms of the gradient
\cite{HMW10}. Dinew and Ko{\l}odziej subsequently proved a Liouville
theorem and used a blow-up argument to obtain the gradient estimate,
completing the smooth existence theory for positive, volume-normalized
data on compact K\"ahler manifolds \cite{DK17}. In the Hermitian setting,
Tosatti and Weinkove solved the complex Monge--Amp\`ere equation with an
unknown normalization constant \cite{TW10}, and Zhang established
the corresponding result for the standard complex $k$-Hessian equation
with the reference Hermitian metric as background \cite{Zhang17}.

The $J$-equation and inverse Hessian equations also arise naturally from
geometric flows. Donaldson's moment-map formulation \cite{Donaldson99}
led to the $J$-flow, whose convergence was characterized by Song and
Weinkove in terms of a positivity condition \cite{Song2004OnTC}.
Fang, Lai and Ma studied a class of fully nonlinear flows extending
this picture to inverse Hessian equations \cite{fang2010classfullynonlinearflow}.
Related eigenvalue equations appear in deformed Hermitian--Yang--Mills
geometry, where Collins, Jacob and Yau obtained a priori estimates and
algebraic obstructions \cite{Collins201511FW}.

A unifying analytic tool is the notion of a $\mathcal C$-subsolution.
Sz\'{e}kelyhidi developed a priori estimates for a general class of fully
nonlinear equations on compact Hermitian manifolds under this condition,
and applied the method to Hessian quotient equations on K\"ahler manifolds
\cite{Sze18}. The higher-regularity theory can be organized using the
Evans--Krylov type estimates of Tosatti, Wang, Weinkove and Yang
\cite{TWWY15}. More recently, Guo and Song introduced an analytic
sup-slope and characterized solvability for a class of concave elliptic
Hessian-type equations by a subsolution condition at that slope
\cite{GS24}. Their work also treats the nonconstant $J$-equation and
clarifies the role of the unknown normalization constant.

In a complementary direction, numerical criteria seek to express
solvability through intersection inequalities. Demailly and P\u{a}un
gave a numerical characterization of the K\"ahler cone
\cite{JPDemaillyPaun}. For the $J$-equation, Collins and
Sz\'{e}kelyhidi established a numerical criterion in the toric setting
\cite{CollinsGabor}, and Chen obtained the corresponding result on
arbitrary compact K\"ahler manifolds \cite{GaoChen2021}.
Related Nakai--Moishezon type criteria have been developed for
generalized Monge--Amp\`ere and Hessian equations
\cite{Datar2020ANC,song2020nakaimoishezoncriterionscomplexhessian},
for the supercritical deformed Hermitian--Yang--Mills equation
\cite{ChuLeeTakhashi}, and for fully nonlinear equations involving
differential forms \cite{Fang2023OnAF}.

Recent work extends these numerical questions to further Hessian and
quotient equations. Fu, Zhang and Zhang proved a numerical criterion
for the complex $2$-Hessian equation in dimension three, with consequences
for the associated Hessian quotient equation \cite{2026arXiv260711296F}.
Chen, Nie and Xu obtained a criterion for complex Hessian type equations
on projective manifolds \cite{CNX26}, while Fu, Li, Zhang and Zhang
studied Nakai--Moishezon type criteria for Hessian type equations on
compact K\"ahler manifolds \cite{fu2026nakaimoishezontypecriteriahessian}.

Beyond individual Hessian and quotient equations, a central problem is
to identify the appropriate elliptic branch for a polynomial combination
of symmetric functions. Lin introduced right-Noetherian polynomials
in the study of general inverse $\sigma_k$ equations and established
convexity results for the associated strictly $\Upsilon$-stable level
hypersurfaces \cite{Lin23,Lin24corr}. He subsequently proved solvability
on compact K\"ahler manifolds under a $\mathcal C$-subsolution condition
and integral compatibility for constant positive-order coefficients,
allowing the zeroth-order coefficient to vary \cite{Lin23solv}.

Fang and Ma placed these structures in the broader theory of
G{\aa}rding polynomials \cite{FM26a}. Their theory of ideal G{\aa}rding
polynomials supplies convex distinguished components, polarization
results, and quotient concavity \cite{FM26b}. These algebraic tools,
together with the subsolution and slope methods above, provide the
background for studying \eqref{e main eq} with position-dependent
coefficients. In particular, they motivate treating the choice of
elliptic branch and the normalization constant as linked parts of the
solvability problem.

\subsection{Hermitian case}
In the Hermitian case, we make the following assumptions.
\begin{assumption}\label{ass1}
    \begin{enumerate}
        \item For any $z\in X$, $p_z$ is strictly right-Noetherian.
        \item $a_i(z)$, $r_i(z)$ are smooth functions for $0 \le i \le d-1$. Here $r_i(z)$ is the largest real root of $p_z^{(i)}.$
        \item $a_{d-1}\ge 0$.
    \end{enumerate}
\end{assumption}
Lemma \ref{lem ass 2} provides  situations where the assumption that $r_i(z)$ are smooth functions can be implied by other assumptions in the Assumption \ref{ass1}. The assumption (iii) is equivalent to a radial condition $\lambda\cdot\nabla P_z(\lambda)>0$ on $\partial C_{P_z}$, which plays an important role in the analysis. We will discuss this condition in section~\ref{sec:radial-condition} in details.

\begin{theorem}\label{thm:radial}
Let $(X,\chi)$ be a compact connected Hermitian manifold and let $\omega_0$  be a smooth real $(1,1)$-form on $X$. Suppose Assumption~\ref{ass1} holds and
$\SH_d(X,\omega_0)\ne\varnothing$. The following are equivalent:
\begin{enumerate}
\item There is $\varphi\in C^\infty(X,\R) \cap \SH_d(X,\omega_0)$, unique modulo constants, and $c>0$ such that
\begin{equation}\label{e thm:radial}
P_z(c^{-1}\Lambda_\varphi)=0,\qquad
c^{-1}\Lambda_\varphi(z)\in \mathcal{C}_{P_z}^{(1)}\quad(z\in X).
\end{equation}
\item There is a $\mathcal C$-subsolution at the radial slope $\sigma$.
\end{enumerate}
If one of the above holds, then $c=\sigma$ and it attains both radial slopes:
\begin{equation}\label{eq:23}
\begin{aligned}
\kappa_z(\Lambda_\varphi)\equiv\sigma 
&=\inf_{v\in\SH_d(X,\omega_0)}\max_X\kappa_z(\Lambda_v)\\
&=\sup_{v\in\SH_d(X,\omega_0)}\min_X\kappa_z(\Lambda_v).
\end{aligned}
\end{equation}
\end{theorem}

The definition of $\SH_d(X,\omega_0)$, $\kappa_z(\Lambda_\varphi)$ and radial slopes are given in Preliminary section. One Key point in the proof of the Theorem \ref{thm:radial} is the concavity that Fang-Ma established in \cite{FM26b}. We generalize such concavity to a joint concavity such that we can deal with the non-constant coefficient case.

We use radial normalization in (\ref{e thm:radial}) in order to answer the following questions:
\begin{enumerate}
    \item The set of admissible solutions to (\ref{e com k hes}) and (\ref{e com quo}) are $$\{\varphi \in C^{\infty}(X): \lambda[\chi^{i \bar{l}}(\omega_{\varphi})_{j \bar{l}}] \in \Gamma_k\}.$$ What is the correct definition of admissible solutions to (\ref{e main eq})?
    \item How should a normalization constant enter \eqref{e com k hes}?
For the complex $k$-Hessian equation and complex Hessian quotient equation, $C$ appears as a solvability constant which admits a min-max characterization by \cite{GS24}. 
\end{enumerate}

The framework used in the Theorem \ref{thm:radial} covers the results on complex $k$-Hessian equations \cite{Sze18,Zhang17} and complex quotient equations \cite{Sze18}. See the section \ref{sec:source-9} for discussion.

\subsection{K\"ahler case}
In the K\"ahler case, we use the following assumption:
\begin{assumption}\label{ass2}
    \begin{enumerate}
        \item For any $z\in X$, $p_z$ is strictly right-Noetherian.
        \item $a_i(z)$ are constant functions for $1\le i \le d-1$, and $a_0(z)$ is a smooth function.
    \end{enumerate}
\end{assumption}

Then, applying the Theorem \ref{thm:radial}, we can reprove the results by Fang-Ma \cite{fang2026spectralgaardingelliptic}:

\begin{theorem}\label{main thm 2}
Let $(X,\chi)$ be a compact connected K\"ahler manifold of complex
dimension $n$, let $2\le d\le n$, and let $\omega_0$ be a smooth
closed real $(1,1)$-form. Suppose Assumption~\ref{ass2} holds and
\begin{equation}\label{eq:kahler-fixed-compatibility}
\int_X\omega_0^d\wedge\chi^{n-d}
=\sum_{k=1}^{d-1}a_k\int_X\omega_0^k\wedge\chi^{n-k}
+\int_Xa_0\chi^n.
\end{equation}
Put $\gamma:=\frac{a_{d-1}}d=r_{d-1}$.
The following are equivalent:
\begin{enumerate}
\item There exists
$\varphi\in\SH_d(X,\omega_0-\gamma\chi)$, unique modulo constants, such that
\[
P_z(\Lambda_\varphi)=0,\qquad
\Lambda_\varphi(z)=\lambda[\chi^{i \bar{j}}(\omega_{\varphi})_{l \bar{j}}]\in \mathcal{C}_{P_z}^{(1)}\quad(z\in X).
\]
\item There exists $\underline\varphi\in C^2(X,\R)$ such that
\[
\Lambda_{\underline\varphi}(z)\in \mathcal{C}_{P_z}^{(1)}
\quad(z\in X).
\]
\end{enumerate}
If $a_{d-1}\ge0$, the solution also belongs to
$\SH_d(X,\omega_0)$.
\end{theorem}

The paper is organized as follows. Section~2 recalls right-Noetherian
polynomials, their polarizations, and the required domination results.
Section~\ref{sec:radial-condition} characterizes the radial condition,
and Section~\ref{sec:geometry} studies the polynomial components and
constructs the radial scale. Section~\ref{sec:joint-concavity}
establishes the joint concavity results used in the analysis.
Section~\ref{sec:concavity} derives the a priori estimates, and
Section~\ref{sec:continuity} proves Theorem~\ref{thm:radial}.
Section~\ref{sec:source-9} treats the Hermitian complex $k$-Hessian equations and the complex quotient equations. The final section discusses the K\"ahler case.

\bigskip

\noindent{\bf Tool and computational resource disclosure:}
ChatGPT 6-Astra was used to perform literature review and as an adversarial reader and improve writing of the arguments. ChatGPT suggested that (\ref{eq:17}) is useful for radial normalization.  The authors retain exclusive responsibility for the correctness, accuracy, and validity of the results presented in this work.

\section{Preliminary}

\subsection{Distinguished component}
Let $\mathcal{C}_P$ denote the connected component of $\{P>0\}$
containing all sufficiently large positive diagonal points.
The first derived component is defined as $\mathcal{C}_{P_z}^{(1)}=\bigcap_{i=1}^n \mathcal{C}_{\partial_iP_z}$. 
A solution is \emph{elliptic} if
$c^{-1}\Lambda_\varphi(z)\in \mathcal{C}_{P_z}^{(1)}$ for every $z$.
Elliptic solution to $P_z(c^{-1}\Lambda_\varphi)=0$ is equivalent to
$c^{-1}\Lambda_\varphi(z)\in\partial \mathcal{C}_{P_z}$.

\subsection{Radial scale}
Under Assumption~\ref{ass1}, the radial scale is defined by
\begin{equation}\label{e radial scale}
\kappa_z(\lambda)>0,\qquad
\kappa_z(\lambda)^{-1}\lambda\in\partial \mathcal{C}_{P_z},
\qquad (z,\lambda)\in X\times\Gamma_d.
\end{equation}
Its existence, uniqueness, smoothness, and concavity are proved in
Lemma~\ref{lem:radial-scale}. Each $\kappa_z$ is homogeneous of degree one
and strictly increasing in every coordinate.

Put
\begin{equation}\label{eq:2}
\SH_d(X,\omega_0)=
\{\varphi\in C^\infty(X,\R):\Lambda_\varphi(z)\in\Gamma_d
\text{ for every }z\in X\}.
\end{equation}
When this class is nonempty, define its radial slope by
\begin{equation}\label{eq:20}
\sigma=\inf_{\varphi\in\SH_d(X,\omega_0)}
\max_X\kappa_z(\Lambda_\varphi),\qquad 0<\sigma<\infty.
\end{equation}
The positivity and finiteness follow from Lemma~\ref{res:7-4}.
The admissible class is independent of $P_z$, whereas
$\kappa$ and $\sigma$ depend on its coefficients.

The following definition uses the bounded-slice notion of \cite[Definition 1]{Sze18}.
\begin{definition}[C-subsolution at the radial slope]\label{def:radial-subsolution}
Assume $\SH_d(X,\omega_0)\ne\varnothing$. A function $\underline u\in C^2(X,\R)$ is a C-subsolution at $\sigma$ if, for every $z\in X$, the set
\begin{equation}
\{\lambda\in\Gamma_d:\kappa_z(\lambda)=\sigma,\quad
\lambda-\Lambda_{\underline u}(z)\in\Gamma_n^+\}
\label{eq:21}
\end{equation}
is bounded.  
\end{definition}
We remark that a $\mathcal{C}$-subsolution $\underline\varphi$ need not be admissible. 

\subsection{G\aa rding polynomial}
\begin{definition}[{\cite[Definitions 2.2 and 2.3]{FM26a}}]
For \(n\ge1\), write
\begin{align*}
  \Gamma_n^+
  &=\{x=(x_1,\ldots,x_n)\in\R^n:x_i>0\text{ for every }i\},\\
  \overline{\Gamma_n^+}
  &=\{x=(x_1,\ldots,x_n)\in\R^n:x_i\ge0\text{ for every }i\}.
\end{align*}
A set \(\mathcal{C} \subseteq\R^n\) satisfies the \emph{positive ray property}
\emph{(PRT)} if
\[
  \mathcal{C}+\overline{\Gamma_n^+}\subseteq \mathcal{C}.
\]
It satisfies the \emph{negative ray property} \emph{(NRT)} if, for
every \(x\in \mathcal{C}\), the set
\[
  (\{x\}-\overline{\Gamma_n^+})\cap \mathcal{C}
\]
is bounded.
\end{definition}

\begin{definition}[Theorem 1.1(3) of
\cite{FM26a}]
\label{FMGdefinition}
Let \(p\in\R[x_1,\ldots,x_n]\) be nonzero.  For a
multi-index
\(\alpha=(\alpha_1,\ldots,\alpha_n)\in\mathbb Z_{\ge0}^n\), write
\[
  \partial^\alpha p
  =
  \frac{\partial^{|\alpha|}p}
       {\partial x_1^{\alpha_1}\cdots\partial x_n^{\alpha_n}},
  \qquad |\alpha|=\alpha_1+\cdots+\alpha_n.
\]
Then \(p\) is called a \emph{Fang-Ma-G{\aa}rding polynomial} if and only if,
for every \(\alpha\) such that \(\partial^\alpha p\not\equiv0\),
\begin{enumerate}[label=\textup{(\roman*)}]
  \item the set \(\{\partial^\alpha p>0\}\) has a unique connected
  component \(\mathcal{C}_{\partial^\alpha p}\) satisfying PRT; and
  \item this component is contained in the corresponding components
  of all its first partial derivatives:
  \[
    \mathcal{C}_{\partial^\alpha p}
    \subseteq \mathcal{C}_{\partial_i\partial^\alpha p}
    \qquad(1\le i\le n).
  \]
\end{enumerate}
Here \(\mathcal{C}_h=\R^n\) when \(h\equiv c>0\), and by convention \(\mathcal{C}_0=\R^n\).
The component \(\mathcal{C}_p\) is called the \emph{Fang-Ma-G{\aa}rding component} of
\(p\), which is a generalization of the $\Upsilon$-stable component of \cite{Lin23}.  We denote the class of such polynomials in \(n\) variables by
\(\Ga_n\), and put \(\Ga=\bigcup_{n\ge1}\Ga_n\).
\end{definition}

\begin{definition}(Definition 1.1 of \cite{FM26b})
Denote by $\mathrm{G}_n^d$ the set of $n$-variate G{\aa}rding
polynomials of degree $d$. We define the ideal G{\aa}rding
polynomials recursively as follows. Set
\[
\mathrm{I}_n^0:=\{f\equiv c:c\geq0\},
\qquad
\mathrm{I}_n^1:=\mathrm{G}_n^1.
\]
For $d\geq2$, let
\[
\mathrm{I}_n^d
:=
\left\{
f:
\partial_i f\in\mathrm{I}_n^{d-1}\ \forall i\in[n],\
\varnothing\neq\mathcal{C}_f
\subset\bigcap_{i=1}^{n}\mathcal{C}_{\partial_i f},\
\mathcal{C}_f\text{ is convex}
\right\}.
\]

\end{definition}

\begin{definition}[{Definitions 2.1} of \cite{Lin23}]
\label{right-Noetherian}
For a nonconstant univariate polynomial \(h\), let \(r(h)\) denote its
largest real root, provided such a root exists.  If
\(\deg h=d\), its \emph{root sequence} is
\[
  R(h)=\bigl(r(h),r(h'),\ldots,r(h^{(d-1)})\bigr),
\]
provided all entries exist.  The polynomial \(h\) is
\emph{right-Noetherian} if all these entries exist
and
\[
  r(h)\ge r(h')\ge\cdots\ge r(h^{(d-1)}).
\]

The polynomial \(h\) is
\emph{strictly right-Noetherian} if all these entries exist
and
\[
  r(h)> r(h')\ge\cdots\ge r(h^{(d-1)}).
\]

The polynomial \(h\) is
\emph{strongly strictly right-Noetherian} if all these entries exist
and
\[
  r(h)>r(h')>\cdots>r(h^{(d-1)}).
\]
\end{definition}

\begin{lemma}[Section 9.2 of \cite{FM26a}]
\label{FangMaLinCorrespondence}
A monic univariate polynomial $p(x)\in \Ga_1$ if and only if it is right-Noetherian.
\end{lemma}

\begin{definition}[Root parameters]
Put
\[
 \K=\{\s=(s_0,\ldots,s_{d-1})\in\R^d:
                  s_0\ge s_1\ge\cdots\ge s_{d-1}\}.
\]
For $e\ge1$, define
\begin{equation}\label{eq:iterated}
 \frp_e(y;a_0,\ldots,a_{e-1})
 =\int_{a_0}^{y}\int_{a_1}^{t_1}\cdots
     \int_{a_{e-1}}^{t_{e-1}}dt_e\cdots dt_1,
 \qquad \frp_0=1.
\end{equation}
The monic root-parameter family is
\begin{equation}\label{eq:family}
 p_{\s}(x)=d!\frp_d(x;s_0,\ldots,s_{d-1}),
 \qquad P_{\s}=\Pol_n(p_{\s}).
\end{equation}
\end{definition}

For $0\le j<d$,
\begin{equation}\label{eq:rootderivatives}
 p_{\s}^{(j)}(x)=d!\frp_{d-j}(x;s_j,\ldots,s_{d-1}),
 \qquad p_{\s}^{(d)}=d!.
\end{equation}
The largest real root of $p_{\s}^{(j)}$ is $s_j$: the expression vanishes
at $s_j$ and is positive for $x>s_j$. Thus $p_{\s}$ is right-Noetherian,
meaning that its largest derivative roots are ordered as above.
Conversely, successive integration reconstructs every monic
right-Noetherian polynomial from this root sequence. If $s_0>s_1$, then $p_{\s}$ is \emph{strictly} right-Noetherian. If $s_0>s_1>s_2>\ldots > s_{d-1}$, then $p_{\s}$ is \emph{strongly strictly} right-Noetherian.

\begin{lemma}\label{lem neg ps}
 Suppose that $\s \in \K$. Then, for any $j$, $p_{\s}^{(j)}< 0$ on $[s_{j+1},s_j).$ 
\end{lemma}
\begin{proof}
Denote $\s^{(j)}=(s_{j},\ldots, s_{d-1})$. Then, we have that:
    \begin{equation}\label{e expre psj}
        p_{\s}^{(j)}(y)= d! \int_{s_j}^{y} \int_{s_{j+1}}^{t_{j+1}}\ldots \int_{s_{d-1}}^{t_{d-1}} dt_d \ldots dt_{j+1}=  \int_{s_j}^y p_{\s^{(j+1)}}^{(j+1)}(t_{j+1}) dt_{j+1}.
    \end{equation}
    Since $s_{j+1}$ is the biggest real root of $p_{\s^{(j+1)}}^{(j+1)}$, we have that $p_{\s^{(j+1)}}^{(j+1)}>0$ on $(s_{j+1},s_j)$. Then, (\ref{e expre psj}) implies that $p_{\s}^{(j)}< 0$ on $[s_{j+1},s_j).$ 
\end{proof}

\begin{definition}[{Definition 10.8} of \cite{FM26a}]\label{def dominate}
Let $g$ be a G\aa rding polynomial. If $f$ is a polynomial satisfies
\begin{equation*}
    \CC_{\partial^{\alpha}g} \subset \{\partial^{\alpha}f>0\} \text{ for every nontrivial partial derivative } \partial^{\alpha} f,
\end{equation*}
then we say that $g$ dominates $f$, and write \(g \triangleright f\).
\end{definition}
The case that $g$ and $f$ are right-Noetherian monic polynomials of degree $n$ is given by [{Definition 1.3} of \cite{Lin23}].

\subsection{Polarization}
Fix integers $1\le d\le n$. For $0\le k\le n$, write
\[
 S_k(\lambda)=\frac{\sigma_k(\lambda)}{\binom nk},
 \qquad S_0=1,\qquad \lambda\in\R^n.
\]
For a polynomial $p(x)=\sum_{k=0}^d c_kx^k$, its normalized polarization is
\begin{equation}\label{eq:polarization}
 \Pol_n(p)(\lambda)=\sum_{k=0}^d c_kS_k(\lambda).
\end{equation}
It is the symmetric multi-affine polynomial with
$\Pol_n(p)(x\one)=p(x)$. Here $\one=(1,\ldots,1)\in\R^n$ and
$\D=\sum_{i=1}^n\partial_{\lambda_i}$. In particular,
\begin{equation}\label{eq:commutation}
 \D^j\Pol_n(p)=\Pol_n(p^{(j)}).
\end{equation}

\begin{theorem}[Theorem 10.10 of \cite{FM26a}]
\label{UpsilonDominance}
Suppose that $g$ is a monic G \aa rding polynomial and $f$ is a monic polynomial, and $g \triangleright f$. Let \(F\) and \(G\) be their polarizations on $\mathbb{R}^n$. Then
\[
       g\triangleright f
       \quad\Longleftrightarrow\quad
       \mathcal{C}_G\subseteq \mathcal{C}_F.
\]
\end{theorem}

The $d=n$ case of Theorem \ref{UpsilonDominance} was proved as Theorem 1.3 of \cite{Lin23}.

\begin{theorem}(Theorem 1.4 of \cite{FM26b})\label{thm ideal pola}
Let $f$ be a G{\aa}rding polynomial. Then $f$ is ideal G{\aa}rding if $f$ is a polarization of a univariate G{\aa}rding polynomial.
\end{theorem}

\subsection{Lift to matrix}

\begin{definition}[Majorization]
Let $x,y\in\mathbb{R}^n$.  We say that $x$ is \emph{majorized} by $y$,
and write $x\prec y$, if
\[
  \sum_{i=1}^k x_i^\downarrow
  \leq
  \sum_{i=1}^k y_i^\downarrow
  \qquad (k=1,\ldots,n-1),
\]
and
\[
  \sum_{i=1}^n x_i
  =
  \sum_{i=1}^n y_i.
\]
\end{definition}

Now we state the Ky Fan--Lidskii inequality:

\begin{theorem}[Theorem~III.4.1, p.~69 of \cite{Bha97}]
\label{Fan-Lidskii}
Let $A,B$ be Hermitian $n\times n$ matrices. Then, for $k=1,\ldots,n$,
\[
  \sum_{i=1}^k \lambda_i^\downarrow(A+B)
  \leq
  \sum_{i=1}^k \lambda_i^\downarrow(A)
  +
  \sum_{i=1}^k \lambda_i^\downarrow(B).
\]
Equality holds when $k=n$.  Equivalently,
\[
  \lambda^\downarrow(A+B)
  \prec
  \lambda^\downarrow(A)+\lambda^\downarrow(B).
\]
Consequently, for every $t\in[0,1]$,
\[
  \lambda^\downarrow\bigl(tA+(1-t)B\bigr)
  \prec
  t\lambda^\downarrow(A)+(1-t)\lambda^\downarrow(B).
\]
\end{theorem}

Now we state Rado's corollary of the Hardy--Littlewood--P\'{o}lya theorem.

\begin{theorem}[Corollary~2.B.3 of \cite{MOA11}]
\label{Rado}
For $x,y\in\mathbb{R}^n$,  $x\prec y$ if and only if $x$ belongs to the convex hull of the coordinate permutations of $y$,
        that is,
        \[
          x\in\conv\{Py:P\text{ is a permutation matrix}\}.
        \]
\end{theorem}
In the next theorem, we refer to \cite{FM26b} for the definition of ideal G\aa rding polynoial.
\begin{theorem}[Theorem~11.1 of \cite{FM26b}]\label{thm pullback ide}
    Take $T: \mathbb{R}^n \rightarrow \mathbb{R}^m, T(x)=Ax+b$, where $A$ doesn't have negative entry or zero row and $b \ge 0$. Then for any Ideal G\aa rding polynomial $f$, we have that $T^* f$ is also an ideal G\aa rding polynomial.
\end{theorem}

\subsection{Situation where (ii) in the Assumption \ref{ass1} holds}

\begin{lemma}\label{lem ass 2}
Suppose that $p_z$ is strictly right-Noetherian for any $z\in X$. Suppose that $a_i(z)$ are smooth functions for $0 \le i \le d-1$.
Suppose that one of the following conditions holds. 
\begin{enumerate}
    \item  For any $z\in X$, $p_z$ is strongly strictly right-Noetherian.
    \item $a_i$ are constant functions for $1\le i \le d-1$. 
\end{enumerate}
Then, we have that $r_i(z)$ are smooth functions for $0 \le i \le d-1$.
\end{lemma}
\begin{proof}
    First, we assume that for any $z\in X$, $p_z$ is strongly strictly right-Noetherian. Then, for any $i\le d-1$, $\frac{d}{ds}p_z^{(i)}(r_i)\neq 0$. By the implicit function theorem, $r_i(z)$ is a smooth function on $X$.

    Next, we assume that $a_i$ are constant functions for $1\le i \le d-1$. Using the argument above and the assumption that $p_z$ is strictly right-Noetherian, we can get that $r_0(z)$ is smooth. Since $a_i$ are constant functions for $1\le i \le d-1$, we have that the coefficients of $p_z^{(i)}$ are constant functions. As a results, $r_i(z)$ are constant functions for $1\le i\le d-1$.
\end{proof}

\section{Relative position between origin and Garding component}\label{sec:radial-condition}
The following radial condition will play a key role in our analysis
\begin{equation}\label{eq:17}
\lambda\cdot\nabla P_z(\lambda)>0
\qquad(z\in X,\ \lambda\in\partial \mathcal{C}_{P_z}).
\end{equation}
For a strictly right-Noetherian polynomial, the following theorem shows
that this is equivalent to the last condition in
Assumption~\ref{ass1}.
\begin{theorem}\label{thm relative position}
Let $p$ be a monic right-Noetherian polynomial of degree $2\le d\le n$, and let $P=\Pol_n(p)$. The following conditions are equivalent:
\begin{enumerate}
\item
\begin{equation}\label{e boundary gradient}
\lambda\cdot\nabla P(\lambda)>0
\qquad\text{for every }\lambda\in\partial \mathcal{C}_P.
\end{equation}
\item
\[
r_0>r_1
\qquad\text{and}\qquad
r_{d-1}\ge0.
\]
\item
The polynomial \(p\) is strictly right-Noetherian and
\[
a_{d-1}\ge0.
\]
\end{enumerate}
Under these conditions, the sharp lower bound is
\[
\min_{\lambda\in\partial C_P}
\lambda\cdot\nabla P(\lambda)
=r_0p'(r_0)>0.
\]
\end{theorem}

\begin{proof}
Since
\[
p^{(d-1)}(s)=d!s-(d-1)!a_{d-1},
\]
we have
\[
r_{d-1}=\frac{a_{d-1}}{d}.
\]
Thus the second and third conditions are equivalent.

\medskip
\noindent
$(ii)\Longrightarrow (i)$. Assume \(r_0>r_1\) and \(r_{d-1}\ge0\). Then
\begin{equation}\label{e nonnegative root}
    r_0>r_1\ge\cdots\ge r_{d-1}\ge0.
\end{equation}

We have that 
\begin{equation}\label{e r0 pr0 pos}
r_0>0, p'(r_0)>0.    
\end{equation}

For \(t>1\), define
\[
q_t(s)=t^dp(s/t),
\qquad
h_t(s)=p(s)-p(tr_0).
\]
Their derivative-root sequences are
\[
R(q_t)=(tr_0,tr_1,\ldots,tr_{d-1}),
\qquad
R(h_t)=(tr_0,r_1,\ldots,r_{d-1}).
\]
Indeed, \(h_t'=p'\), while \(p'>0\) on
\((r_1,\infty)\); hence \(tr_0\) is the largest real root
of \(h_t\).

Both polynomials are monic and right-Noetherian. Using (\ref{e nonnegative root}) and the assumption $t>1$, we have that $q_t \triangleright h_t$. Using Theorem \ref{UpsilonDominance}, we have that:
\[
\mathcal{C}_{\Pol_n(q_t)}\subseteq \mathcal{C}_{\Pol_n(h_t)}.
\]
Since
\[
\Pol_n(q_t)(\lambda)=t^dP(\lambda/t),
\qquad
\Pol_n(h_t)(\lambda)=P(\lambda)-p(tr_0),
\]
this inclusion becomes
\[
t\mathcal{C}_P\subseteq \mathcal{C}_{P-p(tr_0)}.
\]
Passing to closures gives, for every
\(\lambda\in\partial C_P\),
\[
P(t\lambda)\ge p(tr_0).
\]
Both sides vanish at \(t=1\). Taking the right derivative
at \(t=1\), we obtain
\[
\lambda\cdot\nabla P(\lambda)\ge r_0 p'(r_0)>0.
\]

Differentiating \(P(s\one)=p(s)\) gives
\[
(r_0\one)\cdot\nabla P(r_0\one)=r_0p'(r_0).
\]
Therefore the lower bound is sharp and is attained. By (\ref{e r0 pr0 pos}), we have that $r_0p'(r_0)>0$.

$(i)\Longrightarrow (ii)$. Suppose that \(r_0=r_1\), then \(p'(r_0)=0\). At the diagonal
boundary point \(r_0\one\),
\[
(r_0\one)\cdot\nabla P(r_0\one)
=r_0p'(r_0)=0.
\]
This contradicts with (\ref{e boundary gradient}). As a result, we have that $r_0>r_1$.

Set
\[
m=n-d+1,
\qquad
c=\frac{a_{d-1}}d,
\]
and consider
\[
\lambda(T,s)=
(\underbrace{T,\ldots,T}_{d-1},
 \underbrace{s,\ldots,s}_{m}),
\qquad
q_T(s)=P(\lambda(T,s)).
\]
Counting the terms of highest degree in \(T\) gives
\[
q_T(s)
=
\frac{m}{\binom nd}(s-c)T^{d-1}
+O(T^{d-2}),
\]
uniformly in \(C^1\) on every fixed compact
\(s\)-interval. Here we used
\[
\frac{\binom nd}{\binom n{d-1}}=\frac md.
\]
The implicit function theorem, applied with parameter
\(T^{-1}\), therefore supplies a zero
\[
s_T=c+O(T^{-1}).
\]

We must verify that this zero belongs to the
distinguished boundary.
Fix \(\delta>0\) and choose
\[
B>\max\{r_0,c+2\delta\}.
\]
For sufficiently large \(T\), the preceding expansion
shows that \(q_T'>0\) on
\([c-\delta,c+\delta]\) and \(q_T>0\) on
\([c+\delta,B]\). Hence
\[
q_T(s)>0\qquad(s_T<s\le B).
\]
For \(T\ge B\), the point \(\lambda(T,B)\) belongs to
\(\mathcal{C}_P\), because $B > r_0$ and $\CC_P$ satisfies PRT.
Thus the positive path
\[
\{\lambda(T,s):s_T<s\le B\}
\]
lies in \(\mathcal{C}_P\), and its zero endpoint satisfies
\[
\lambda_T:=\lambda(T,s_T)\in\partial \mathcal{C}_P.
\]

Euler's identity, together with \(P(\lambda_T)=0\),
gives
\[
\lambda_T\cdot\nabla P(\lambda_T)
=
\sum_{k=0}^{d-1}(d-k)a_kS_k(\lambda_T).
\]
Since \(s_T\) remains bounded,
\[
S_{d-1}(\lambda_T)
=
\frac{T^{d-1}}{\binom n{d-1}}+O(T^{d-2}),
\qquad
S_k(\lambda_T)=O(T^k)
\quad(k\le d-2).
\]
Consequently,
\[
\lambda_T\cdot\nabla P(\lambda_T)
=
\frac{a_{d-1}}{\binom n{d-1}}T^{d-1}
+O(T^{d-2}).
\]
If \(a_{d-1}<0\), this quantity is negative for all
sufficiently large \(T\), contradicting the desired
boundary positivity. This proves that $a_{d-1}\ge 0$.
\end{proof}

\begin{corollary}[Parameter-dependent version]
Let \(p_z\) be a smooth family of monic strictly
right-Noetherian polynomials of fixed degree
\(2\le d\le n\), parametrized by a compact manifold \(M\).
Let \(P_z=\Pol_n(p_z)\).

Then
\[
\lambda\cdot\nabla_\lambda P_z(\lambda)>0
\qquad
\bigl(z\in M,\ \lambda\in\partial \mathcal{C}_{P_z}\bigr)
\]
holds if and only if
\[
a_{d-1}(z)\ge0\qquad(z\in M).
\]
Under this condition,
\[
\min_{\substack{z\in M\\
                \lambda\in\partial C_{P_z}}}
\lambda\cdot\nabla_\lambda P_z(\lambda)
=
\min_{z\in M}r_0(z)p_z'(r_0(z))>0.
\]
\end{corollary}

\begin{proof}
Apply the theorem in each fiber.
Strict right-Noetherianity makes the largest root
\(r_0(z)\) simple. The implicit function theorem,
together with continuity of polynomial roots, makes
\(r_0(z)\) smooth.
Thus
\[
z\longmapsto r_0(z)p_z'(r_0(z))
\]
is a positive continuous function on the compact
manifold \(M\), so its minimum is positive.
The equality follows from the diagonal equality in
each fiber. No smoothness of the lower derivative
roots is required.
\end{proof}

\section{Polynomial components and the radial scale}\label{sec:geometry}
We work on a compact connected Hermitian manifold $(X,\chi)$, with
$\omega_0$ a smooth real $(1,1)$-form. The polynomial family
satisfies Assumption~\ref{ass1}(i)--(ii). The radial condition
\eqref{eq:17} is imposed where stated.

\begin{proposition}\label{prop:derived-component}
The domain $ \mathcal{C}_{P_z}^{(1)}$ is open, convex, symmetric, and stable under addition of $\overline{\Gamma_n}$. It is the connected component of $\{\lambda:\partial_iP_z(\lambda)>0\ (1\le i\le n)\}$ containing sufficiently large diagonal points. Moreover, \begin{equation}\label{eq:7}
\mathcal{C}_{P_z}^{(1)}=\bigcap_{i=1}^n \mathcal{C}_{\partial_iP_z}
=\{\lambda\in\R^n:\partial_IP_z(\lambda)>0,
\ 1\le |I|\le d-1\},
\end{equation}
\begin{equation}
 \mathcal{C}_{P_z}= \mathcal{C}_{P_z}^{(1)}\cap\{P_z>0\}.
\label{eq:25}
\end{equation}

\end{proposition}

\begin{proof}
By \cite[Definition 1.1 and Theorem 1.4(2)]{FM26b}, $P_z$ and its nonzero derivatives are ideal G\aa rding. Their distinguished components are open, convex, and stable under addition of $\Gamma_n$, with
\[
 \mathcal{C}_{P_z}\subset \mathcal{C}_{P_z}^{(1)}\subset\{\lambda:\partial_IP_z(\lambda)>0,\ 1\le|I|\le d-1\}.
\]
Here we write
$[n]=\{1,\ldots,n\}$, and set $\partial_I=\prod_{i\in I}\partial_{\lambda_i}$.
Multi-affinity gives
\begin{equation}
\partial_JP_z(\lambda+y)=\sum_{I\subset[n]\setminus J}\partial_{I\cup J}P_z(\lambda)\prod_{i\in I}y_i
\qquad(J\subset[n]),
\label{eq:26}
\end{equation}
where $\partial_\varnothing P_z=P_z$, $\partial_IP_z=\binom nd^{-1}$ for $|I|=d$, and $\partial_IP_z=0$ for $|I|>d$. Suppose $\partial_IP_z(\lambda)>0$ for $1\le|I|\le d-1$. Choose $T>\max_i\lambda_i$ sufficiently large and put $\gamma(t)=\lambda+t(T\one-\lambda)$. Then
\[
\partial_iP_z(\gamma(t))>0\quad(0\le t\le1),\qquad
\gamma(1)=T\one\in \mathcal{C}_{\partial_iP_z}\quad\Longrightarrow\quad\lambda\in \mathcal{C}_{\partial_iP_z}.
\]
This proves \eqref{eq:7}. If also $P_z(\lambda)>0$, \eqref{eq:26} with $J=\varnothing$ gives $P_z(\gamma(t))>0$; hence $\lambda\in \mathcal{C}_{P_z}$, proving \eqref{eq:25}. Let $\Omega$ be the component of $\bigcap_i\{\partial_iP_z>0\}$ containing $T\one$. Connectedness of $\Omega$ and convexity of $ \mathcal{C}_{P_z}^{(1)}$ give, respectively,
\[
\Omega\subset\bigcap_i \mathcal{C}_{\partial_iP_z}= \mathcal{C}_{P_z}^{(1)}\subset\Omega.
\]
Permutation invariance of $P_z$ gives symmetry of $ \mathcal{C}_{P_z}$ and $ \mathcal{C}_{P_z}^{(1)}$; their upward stability follows from \eqref{eq:26}.
\end{proof}

\begin{lemma}\label{lem:regular-components}
For every $z\in X$,
\begin{equation}\label{eq:29}
\max_{\partial C_{P_z}^{(1)}}P_z=p_z(r_1(z))<0,
\end{equation}
with equality at $r_1(z)\one$.
\end{lemma}
\begin{proof}
Fix $z$, and write $p=p_z$, $P=P_z$, and $b=p(r_1)$.
Since $p'>0$ on $(r_1,\infty)$ and $r_0>r_1$, one has
$b<p(r_0)=0$.
The root sequence of $p-b$ is $(r_1,r_1,r_2,\ldots,r_{d-1})$.
Proposition~\ref{prop:derived-component}, which also holds without a strict first root
inequality, gives
\[
 \mathcal{C}_{P-b}=\{P>b\}\cap \mathcal{C}_P^{(1)}.
\]
If $\lambda\in\partial \mathcal{C}_P^{(1)}$ and $P(\lambda)>b$, an
approximating sequence from $ \mathcal{C}_P^{(1)}$ eventually lies in
$ \mathcal{C}_{P-b}$. Since a component is closed relative to $\{P>b\}$,
this would imply $\lambda\in \mathcal{C}_{P-b}$, a contradiction. Hence
$P\le b$ on $\partial \mathcal{C}_P^{(1)}$. Note that $P(r_1(z)\one)=p_z(r_1(z))=0$ and $r_1(z)\one \in \partial \mathcal{C}_{P_z}^{(1)}$. This proves \eqref{eq:29}.
\end{proof}

\begin{lemma}\label{lem:component-closure}
For every $z\in X$:
\begin{enumerate}
\item
\begin{equation}\overline{ \mathcal{C}_{P_z}}=\{\lambda:P_z(\lambda)\ge0,\ \partial_IP_z(\lambda)\ge0\ (1\le|I|\le d-1)\}.\label{eq:174}\end{equation}
\item
\begin{equation}\overline{ \mathcal{C}_{P_z}}\subset \mathcal{C}_{P_z}^{(1)},\qquad\partial C_{P_z}=\{\lambda\in \mathcal{C}_{P_z}^{(1)}:P_z(\lambda)=0\}.\label{eq:175}\end{equation}
In particular, $\partial \mathcal{C}_{P_z}$ is smooth and $\partial_iP_z>0$ on it.
\end{enumerate}
\end{lemma}
\begin{proof}
(i). Combining (\ref{eq:7}) and (\ref{eq:25}), we have that: $\mathcal{C}_{P_z}=\{\lambda: P_z(\lambda)>0,\ \partial_IP_z(\lambda)>0\ (1\le|I|\le d-1)\}$. Taking the closure of this formula, we can get (\ref{eq:174}).
(ii). Since $\overline{\mathcal{C}_{P_z}}\subset\overline{\mathcal{C}_{P_z}^{(1)}}$
and $P_z\ge0$ on $\overline{\mathcal{C}_{P_z}}$, the boundary estimate
\eqref{eq:29} gives $\overline{\mathcal{C}_{P_z}}\subset \mathcal{C}_{P_z}^{(1)}$.
At every boundary point of $\mathcal{C}_{P_z}$, we have $P_z=0$. Conversely, if
$\lambda\in \mathcal{C}_{P_z}^{(1)}$ and $P_z(\lambda)=0$, then
$\lambda+\varepsilon\one\in \mathcal{C}_{P_z}$ for every $\varepsilon>0$
by \eqref{eq:26}. This proves \eqref{eq:175}.
\end{proof}

\begin{proposition}\label{prop:radial-nesting}
Suppose Assumption~\ref{ass1}(i)--(iii) hold. For every $z$:
\begin{enumerate}
\item $\overline{ \mathcal{C}_{P_z}}\subset\Gamma_d$.
\item $t\overline{ \mathcal{C}_{P_z}}\subset s C_{P_z}$ for $0<s<t$.
\item
\begin{equation}\bigcup_{t>0}t \mathcal{C}_{P_z}=\Gamma_d,\qquad\bigcap_{t>0}t\overline{ \mathcal{C}_{P_z}}=\varnothing.\label{eq:176}\end{equation}
\end{enumerate}
\end{proposition}
\begin{proof}
Fix $z$ and write $P=P_z$.

\emph{Step 1: outward invariance.} If the ray through $\lambda\in C_P$ first exits at $s_*\lambda$, $s_*>1$, then
\[P(s\lambda)>0\ (1\le s<s_*),\qquad P(s_*\lambda)=0,\qquad s_*\lambda\cdot\nabla P(s_*\lambda)\le0,\]
contradicting Theorem \ref{thm relative position}. Hence
\begin{equation}s \mathcal{C}_P\subset \mathcal{C}_P\qquad(s\ge1).\label{eq:177}\end{equation}
For $\lambda\in \mathcal{C}_P$ and $0\le j\le d-1$,
\[\partial_IS_d(\lambda)=\lim_{s\to\infty}s^{|I|-d}\partial_IP(s\lambda)\ge0,\qquad\binom djS_{d-j}(\lambda)=\sum_{|I|=j}\partial_IS_d(\lambda)\ge0.\]
Thus $ \mathcal{C}_P\subset\overline{\Gamma_d}$; openness gives $ \mathcal{C}_P\subset\Gamma_d$. If $x\in\overline{ \mathcal{C}_P}\cap\partial\Gamma_d$, a supporting normal $\nu\ne0$ for the convex cone $\Gamma_d$ satisfies $\nu\cdot x=0$. It supports $ \mathcal{C}_P$ at $x$ as well. The smoothness in Lemma~\ref{lem:component-closure} implies that $\nu$ is parallel to $\nabla P(x)$, contradicting \eqref{eq:17}. Therefore
\begin{equation}\overline{ \mathcal{C}_P}\subset\Gamma_d.\label{eq:178}\end{equation}
\emph{Step 2: strict nesting.} For $x\in\partial \mathcal{C}_P$,
\[\left.\frac d{ds}\right|_{s=1}P(sx)=x\cdot\nabla P(x)>0.\]
Consequently $(1+\varepsilon)x\in \mathcal{C}_P$ for small $\varepsilon>0$. Together with \eqref{eq:177}, this gives $s\overline{ \mathcal{C}_P}\subset \mathcal{C}_P$ for $s>1$, proving (ii).

\emph{Step 3: union and intersection.} If $\lambda\in\Gamma_d$, all nonzero mixed derivatives of $S_d$ are positive at $\lambda$. Since
\[R^{|I|-d}\partial_IP(R\lambda)\longrightarrow\partial_IS_d(\lambda)>0\qquad(0\le|I|\le d-1),\]
Proposition~\ref{prop:derived-component} gives $R\lambda\in \mathcal{C}_P$ for large $R$. This proves the union in (iii). At $r_0\one\in\partial \mathcal{C}_P$, the Euler condition gives $r_0p'(r_0)>0$, hence $r_0>0$. Averaging coordinate permutations in the convex component gives $S_1(\lambda)>r_0$ on $ \mathcal{C}_P$. Thus
\[t\overline{ \mathcal{C}_P}\subset\{S_1\ge tr_0\},\qquad\bigcap_{t>0}t\overline{ \mathcal{C}_P}=\varnothing.\]
\end{proof}

\begin{lemma}\label{lem:radial-scale}
Under Assumption~\ref{ass1}(i)--(iii), the scale $\kappa_z(\lambda)$ in \eqref{e radial scale} exists and is unique.
\begin{enumerate}
\item $\kappa$ is smooth and positive on $X\times\Gamma_d$.
\item Each $\kappa_z$ is homogeneous of degree one, strictly increasing in each coordinate, and concave.
\item For $t>0$,
\begin{equation}\{\kappa_z>t\}=t \mathcal{C}_{P_z},\qquad\{\kappa_z\ge t\}=t\overline{ \mathcal{C}_{P_z}}.\label{eq:179}\end{equation}
\end{enumerate}
\end{lemma}
\begin{proof}
\emph{Step 1: existence and differentiation.} Proposition~\ref{prop:radial-nesting} gives, for $\lambda\in\Gamma_d$,
\[\{t>0:\lambda\in t \mathcal{C}_{P_z}\}=(0,b),\qquad0<b<\infty,\qquad\lambda/b\in\partial \mathcal{C}_{P_z}.\]
Strict nesting makes $b$ unique. This defines $\kappa_z(\lambda)=b$ and proves \eqref{eq:179} and homogeneity. For $x=\lambda/b$,
\[\partial_bP_z(\lambda/b)=-\frac{x\cdot\nabla P_z(x)}b<0,\qquad\partial_i\kappa_z(\lambda)=\frac{\partial_iP_z(x)}{x\cdot\nabla P_z(x)}>0.\]
The implicit-function theorem gives joint smoothness.

\emph{Step 2: concavity.} Let $a=\kappa_z(\lambda)$, $b=\kappa_z(\mu)$, $0\le\theta\le1$, and $q=\theta a+(1-\theta)b$. Convexity gives
\[\frac{\theta\lambda+(1-\theta)\mu}q=\frac{\theta a}q\frac\lambda a+\frac{(1-\theta)b}q\frac\mu b\in\overline{ C_{P_z}}.\]
By \eqref{eq:179}, $\kappa_z(\theta\lambda+(1-\theta)\mu)\ge q$.
\end{proof}

\begin{lemma}\label{res:7-4}
If $\SH_d(X,\omega_0)\ne\varnothing$, then $0<\sigma<\infty$ and
\begin{equation}\sup_{v\in\SH_d(X,\omega_0)}\min_X\kappa_z(\Lambda_v)\le\sigma,\label{eq:180}\end{equation}
and for any $v_0 \in \SH_d(X,\omega_0)$, we have that $0<\min_X\kappa_z(\Lambda_{v_0})\le\sigma\le\max_X\kappa_z(\Lambda_{v_0})<\infty$.
\end{lemma}
\begin{proof}
For $\varphi,v\in\SH_d(X,\omega_0)$ and $x\in\operatorname{argmax}_X(\varphi-v)$, eigenvalue monotonicity gives
\[\omega_\varphi(x)\le\omega_v(x),\]
Combining this with Lemma \ref{lem:radial-scale}, we can get 
$$\min_X\kappa_z(\Lambda_\varphi)\le\kappa_x(\Lambda_\varphi(x))\le\kappa_x(\Lambda_v(x))\le\max_X\kappa_z(\Lambda_v).$$
Take the supremum in $\varphi$ and infimum in $v$. A fixed $v_0$ gives $0<\min_X\kappa_z(\Lambda_{v_0})\le\sigma\le\max_X\kappa_z(\Lambda_{v_0})<\infty$.
\end{proof}

\section{Joint concavity}\label{sec:joint-concavity}

We first recall the following Fang-Ma's fixed-parameter facts about concavity. Convexity and concavity when the
parameters vary are established separately below.

\begin{lemma}\label{lem:univariate}
For ordered $a_0\ge\cdots\ge a_{e-1}$, the function
$\frp_e^{1/e}$ is jointly concave on $y>a_0$. If $e\ge2$, then
\begin{equation}\label{eq:phi}
 \Phi_e(y;a_0,\ldots,a_{e-1})
       =\frac{\frp_e(y;a_0,\ldots,a_{e-1})}
              {\partial_y\frp_e(y;a_0,\ldots,a_{e-1})}
\end{equation}
is jointly concave and nonincreasing with respect to $a_i$ for each $i$ on $y>a_1$.
For $e=1$, $\Phi_1(y;a_0)=y-a_0$ is affine on all of $\R^2$.
\end{lemma}

These are the results in \cite[Corollary 5.3 and Theorems 6.6--6.7]{FM26b}. Notice that $\Phi_e$ can be negative: its concavity domain
is determined by the derivative, not by positivity of the numerator.

\begin{lemma}\label{lem:mean}
For $0<\theta<1$, $h(u,v)=u^\theta v^{1-\theta}$ is increasing in each
argument and concave on $(0,\infty)^2$. Consequently, if $U,V$ are
positive concave functions on a convex domain, then
$U^\theta V^{1-\theta}$ is concave there.
\end{lemma}
\begin{proof}
The first derivatives of $h$ are positive, and
\[
 D^2h(u,v)[(\alpha,\beta),(\alpha,\beta)]
 =-\theta(1-\theta)h(u,v)
       \left(\frac{\alpha}{u}-\frac{\beta}{v}\right)^2\le0.
\]
For the composition assertion, first use the concavity of $U,V$ and
monotonicity of $h$, and then the concavity of $h$.
\end{proof}

\begin{lemma}\label{lem:quotient}
Let $R(\tau,z,y)$ be a polynomial with respect to $(z,y)$ for each fixed $\tau\in K$ for a convex
parameter set $K$ and $(z,y)\in\R^m\times\R$. Denote $R_{\tau}(z,y)\triangleq R(\tau, z, y)$. Here $R_{\tau}$ is viewed as an element of $\mathbb{R}[z_1,\ldots,z_m][y].$ Suppose that
\begin{enumerate}
\item There exists $\ell\ge 1$ such that $\deg_yR_\tau=\ell$ for every $\tau$. 
\item every fixed-parameter polynomial $R_\tau$ is ideal G\aa rding;
\item $\mathcal{C}_{\partial_y^kR}\triangleq \bigcup_{\tau \in K}(\{\tau\}\times \CC_{\partial_y^k R_{\tau}})$ is convex for every $0\le k\le\ell$.
\end{enumerate}
Then $R/\partial_yR$ is jointly concave on $\CC_{\partial_yR}$.
\end{lemma}

\begin{proof}
Write $\xi=(\tau,z)$ and let $D$ be the projection of
$\CC_{\partial_yR}$ onto the $\xi$-space. The set $D$ is convex. Write
\[
 R(\xi,y)=\sum_{j=0}^{\ell}a_j(\xi)y^j.
\]
The derivative-component inclusions give
\begin{equation}\label{eq:leading}
 \ell!a_\ell(\xi)=\partial_y^\ell R(\xi,y)>0
 \qquad(\xi\in D).
\end{equation}

\emph{Step 1: every fiber over $D$ meets $\CC_R$.}
Fix $\xi=(\tau,z)\in D$, choose $(z_*,y_*)\in \mathcal{C}_{R_\tau}$, and set
$z_t=(1-t)z_*+tz$, $0\le t\le1$.
The derivative-component inclusion puts $(\tau,z_*)$ in $D$. This implies that
$(\tau,z_t)\in D$ because that $D$ is convex. Continuity on this compact segment and
\eqref{eq:leading} give a positive lower bound for $a_\ell$ and a uniform
bound for all coefficient ratios. The Cauchy root bound therefore gives
\begin{equation}\label{eq:cauchy}
 |w|\le1+\max_{j<\ell}
       \left|\frac{a_j(\tau,z_t)}{a_\ell(\tau,z_t)}\right|
       \le C
\end{equation}
for every complex root $w$ of $R(\tau,z_t,\cdot)$ and every $t\in[0,1]$.
Choose $Y>\max\{C,y_*\}$. The leading coefficient is positive and there
is no root above $C$, so $R(\tau,z_t,Y)>0$. Upward stability gives
$(z_*,Y)\in \mathcal{C}_{R_\tau}$. The path $(z_t,Y)$ stays in $\{R_\tau>0\}$ and
therefore remains in this connected component. Its endpoint proves the
claim.

\emph{Step 2: the largest derivative roots are convex functions.}
The restriction property now applies to every fiber over $D$.
The polynomial $R(\xi,\cdot)$ is right-Noetherian, with largest derivative
roots
\[
 \zeta_0(\xi)\ge\zeta_1(\xi)\ge\cdots\ge\zeta_{\ell-1}(\xi).
\]
For $0\le k<\ell$, the derivative-component inclusions and Step 1 give
a nonempty vertical fiber in $\CC_{\partial_y^kR}$ over each $\xi\in D$.
This fiber is upward stable. It cannot contain a point with
$y\le\zeta_k(\xi)$, since then it would contain a zero of
$\partial_y^kR$. Conversely, any point with $y>\zeta_k(\xi)$ connects
vertically, through positive values, to a sufficiently high point already
in the distinguished component. Consequently,
\begin{equation}\label{eq:epigraph}
 \CC_{\partial_y^kR}\cap(D\times\R)
   =\{(\xi,y):\xi\in D,\ y>\zeta_k(\xi)\}.
\end{equation}
The left side is convex. Hence $\zeta_k$ is convex.

\emph{Step 3: compose with the univariate quotient.}
Successive integration gives
\[
 R(\xi,y)=\ell!a_\ell(\xi)
       \frp_\ell(y;\zeta_0(\xi),\ldots,\zeta_{\ell-1}(\xi)).
\]
The prefactor cancels in $R/\partial_yR$. For two points
$(\xi_q,y_q)\in\CC_{\partial_yR}$, put
\[
 (\xi_t,y_t)=(1-t)(\xi_0,y_0)+t(\xi_1,y_1),\qquad
 \bar\zeta=(1-t)\zeta(\xi_0)+t\zeta(\xi_1).
\]
Coordinatewise convexity gives $\zeta(\xi_t)\le\bar\zeta$.
If $\ell\ge2$, then $y_t>\bar\zeta_1$. The root monotonicity and joint
concavity from Lemma~\ref{lem:univariate} give
\begin{align*}
 \frac{R}{\partial_yR}(\xi_t,y_t)
 &=\Phi_\ell(y_t;\zeta(\xi_t))\\
 &\ge\Phi_\ell(y_t;\bar\zeta)\\
 &\ge(1-t)\Phi_\ell(y_0;\zeta(\xi_0))
       +t\Phi_\ell(y_1;\zeta(\xi_1)).
\end{align*}
For $\ell=1$, the quotient is $y-\zeta_0(\xi)$, which is concave because
$\zeta_0$ is convex.
\end{proof}

\subsection{Joint concavity for matrix lifting}
\begin{proposition}\label{prop:matrix}
Suppose $f(\lambda,\tau)$ is jointly concave on a convex domain
$\E\subset\R^n\times K$, and that the domain and the function are
invariant under permutations of $\lambda$. Then
\[
 \widehat{\E}
   =\{(A,\tau):(\lambda(A),\tau)\in\E\}
\]
is convex and $\widehat f(A,\tau)=f(\lambda(A),\tau)$ is jointly concave
on it.
\end{proposition}
\begin{proof}
Let $A_t=(1-t)A_0+tA_1$, $\tau_t=(1-t)\tau_0+t\tau_1$, and set
$\eta_t=(1-t)\lambda(A_0)+t\lambda(A_1)$.
The eigenvalue majorization inequality is
\begin{equation}\label{eq:majorization}
 \lambda(A_t)\prec\eta_t,
 \quad\hbox{equivalently}\quad
 \lambda(A_t)\in\conv\{\pi\eta_t:\pi\text{ a permutation}\},
\end{equation}
according to Theorem \ref{Rado}.

By convexity, $(\eta_t,\tau_t)\in\E$. Symmetry and convexity
then give $(\lambda(A_t),\tau_t)\in\E$. Applying concavity to
the convex combination of permutations in \eqref{eq:majorization} yields
\begin{align*}
 f(\lambda(A_t),\tau_t)
 &\ge f(\eta_t,\tau_t)\\
 &\ge(1-t)f(\lambda(A_0),\tau_0)
          +tf(\lambda(A_1),\tau_1).
\end{align*}
\end{proof}

\subsection{Joint concavity for polarization}\label{sec:rootproof}

We use the notation in the Preliminary section.
Start with $R_0(\s,y)=p_{\s}(y)$ and successively separate variables:
\begin{equation}\label{eq:partial}
 R_j(\s,z_1,\ldots,z_j,y)
 =R_{j-1}+\frac{z_j-y}{n-j+1}\partial_yR_{j-1},
 \qquad1\le j\le n.
\end{equation}
Equivalently,
\begin{equation}\label{eq:partial-diagonal}
 R_j=P_{\s}(z_1,\ldots,z_j,
                \underbrace{y,\ldots,y}_{n-j\ \mathrm{copies}}).
\end{equation}
Indeed, before separating $z_j$ there are $n-j+1$ identical slots equal
to $y$. Symmetry makes their first derivatives equal, and multi-affinity
gives exactly \eqref{eq:partial}. Thus $R_n=P_{\s}(z_1,\ldots,z_n)$.
For fixed $\s$, each $R_j$ is ideal G\aa rding by the partial-polarization
input. The subscript in $(R_j)_{\s}$ means that $\s$ is fixed, not that
a derivative in $\s$ is being taken. Define the total component of $\partial^i_y R_j$ as $$\mathcal{C}_{\partial^i_y R_j} = \{(\s,z_1,\ldots, z_j,y): \s \in \K, (z_1,\ldots, z_j, y)\in \mathcal{C}_{(\partial_y^i R_j)_s}\}$$

\begin{proposition}\label{prop concave pol}
   Suppose that $p_s$ is right-Noetherian. Then, for any $1\le j \le n$, and for  every nonzero $\partial_y^iR_j$,  $\mathcal{C}_{\partial^i_y R_j}$ is convex and, when $i<d$, that
\begin{equation}\label{eq:induction}
 (\partial_y^iR_j)^{1/(d-i)}
 \quad\hbox{is jointly concave on that total component.}
\end{equation}
\end{proposition}
\begin{proof}
We prove by induction on $j$.
For $j=0$, equation~\eqref{eq:rootderivatives} gives
\[
 \partial_y^iR_0=d!\frp_{d-i}(y;s_i,\ldots,s_{d-1}),\qquad
 \CC_{\partial_y^iR_0}=\{(\s,y):\s\in\K,\ y>s_i\}
 \quad(i<d).
\]
Lemma~\ref{lem:univariate} proves the assertion. The derivative of order
$d$ is the positive constant $d!$.

Assume all these assertions at stage $j-1$. Put $m=n-j+1$.
The highest homogeneous part with respect to $(z_1,\ldots, z_j, y)$ in
\eqref{eq:partial-diagonal} is
\[
 \binom nd^{-1}\sum_{k=0}^{\min(j,d)}
    \binom{n-j}{d-k}\sigma_k(z_1,\ldots,z_j)y^{d-k},
\]
where a binomial coefficient is zero when its lower index is outside its
usual range. It follows that
\[
 \deg_y(R_j)_{\s}=\min(d,n-j).
\]
Here the degree is taken with coefficients polynomial in the $z$ variables.
Choose a nonzero $S=\partial_y^iR_j$, and write
$G=\partial_y^iR_{j-1}$ and $e=d-i$. Since
$i\le\min(d,m-1)$, one has $i<m$. Differentiating \eqref{eq:partial}
gives
\begin{equation}\label{eq:RS}
 S=\frac{m-i}{m}G+\frac{z_j-y}{m}\partial_yG.
\end{equation}
In detail, the only extra term in Leibniz' rule is
\[
 \partial_y^i\big((z_j-y)\partial_yR_{j-1}\big)
 =(z_j-y)\partial_y^{i+1}R_{j-1}
                  -i\partial_y^iR_{j-1},
\]
which explains the coefficient $m-i$.

If $\partial_y G\equiv0$, then $S=(m-i)G/m$ with a positive constant
factor. The domain is the old one times the unused $z_j$-line, and the
inductive conclusion follows. This also handles any surviving derivative
of order $i=d$, which is a positive constant.

Suppose now that $\partial_y G\not\equiv0$. All its successive
$y$-derivatives are derivatives of $R_{j-1}$, so their total components
are convex by the induction hypothesis. The fixed-parameter
ideal G\aa rding property and the fixed positive $y$-degree allow
Lemma~\ref{lem:quotient} to be applied to $G$. Hence
\begin{equation}\label{eq:L}
 L=z_j-y+(m-i)\frac{G}{\partial_y G}
 \quad\hbox{is concave on }\CC_{\partial_y G}\times\R_{z_j}.
\end{equation}
Moreover,
\begin{equation}\label{eq:factorS}
 S=\frac{\partial_y G}{m}L,\qquad
 \partial_{z_j}S=\frac{\partial_y G}{m}.
\end{equation}
We claim that
\begin{equation}\label{eq:domainS}
 \CC_S=\{(\xi,y,z_j):(\xi,y)\in\CC_{\partial_y G},\ L>0\},
 \qquad \xi=(\s,z_1,\ldots,z_{j-1}).
\end{equation}
Proof of this claim: Denote $E_{\s}= \{(z_1,\ldots, z_j,y): (\s, z_1,\ldots, z_j,y) \in \CC_{\partial_y G},\ L(\s, z_1,\ldots, z_j,y)>0 \}.$\\
First, we prove that $\{\s\} \times E_{\s}\subset \CC_S$. By (\ref{eq:L}), $L$ is concave. By induction hypothesis, $\CC_{\partial_y G}$ is also convex. As a result, $E_{\s}$ is convex. This implies that $E_{\s}$ is connected.   Consider $p_T$ which corresponds to the point where all the components are equal to $T$. Since $G$ and $\partial_y G$ are G\aa rding polynomials, we have that $(\s,p_T)$ lies in $\CC_G \cap \CC_{\partial_y G}$ for $T$ big enough. As a result, $L(\s,p_T)= (m-i)\frac{G}{\partial_y G}(\s,p_T)>0$. This implies that $p_T \in E_{\s}$ for all $T$ big enough. By (\ref{eq:factorS}), $E_{\s}\subset \{S>0\}.$ Recall that $E_{\s}$ is connected. Then, $E_{\s}$ is contained in the connected component of $\{S>0\}$ which contains $p_T$ for $T$ big enough. By the definition of $\CC_S$, we have that $\{\s\}\times E_{\s} \subset \CC_S$. \\
Next, we prove that $\CC_S \subset \{(\xi,y,z_j):(\xi,y)\in\CC_{\partial_y G},\ L>0\}.$ The derivative-component
inclusion for $S$ gives
$\mathcal{C}_S\subset \mathcal{C}_{\partial_{z_j}S}=\mathcal{C}_{\partial_y G}\times\R_{z_j}$. As a result, for any $(\xi,y,z_j)\in \CC_S$, we have that $(\xi,y)\in \mathcal{C}_{\partial_y G}$. This implies that $S(\xi,y,z_j)>0$ and $\partial_y G(\xi,y)>0$. Then, \eqref{eq:factorS}  implies $L(\xi,y,z_j)>0$. As a result, we have that $(\xi,y,z_j) \in  \{(\xi,y,z_j):(\xi,y)\in\CC_{\partial_y G},\ L>0\}$. This finishes the proof of this Claim.

If $e\ge2$, the induction hypothesis at derivative order $i+1$ says
that $(\partial_y G)^{1/(e-1)}$ is concave. On \eqref{eq:domainS},
\begin{equation}\label{eq:weightedS}
 S^{1/e}=m^{-1/e}
   \left((\partial_y G)^{1/(e-1)}\right)^{(e-1)/e}L^{1/e}.
\end{equation}
Both arguments of the weighted geometric mean are positive and concave;
Lemma~\ref{lem:mean} proves the desired concavity. If $e=1$, $S$ is
affine in all the variables: $p_{\s}(y)$ is homogeneous of degree $d$
jointly in $(\s,y)$, partial polarization preserves this total degree,
and $i=d-1$ differentiations leave degree one.

This completes the induction and proves the assertion for $P_{\s}=R_n$.
Finally,
\begin{equation}\label{eq:directional-family}
 \D^jP_{\s}
   =d!\Pol_n\big(\frp_{d-j}(\,\cdot\,;s_j,\ldots,s_{d-1})\big).
\end{equation}
Apply the same proof to this lower-degree family, retaining the same
$n$ polarization variables. The unused parameters $s_0,\ldots,s_{j-1}$
and the restriction to the convex chamber $\K$ do not affect concavity.
For $j=d$ the polynomial is the constant $d!$.
\end{proof}

\subsection{Joint concavity for quotient}

\begin{proposition}\label{prop:Theta}
The function
\begin{equation}\label{eq:Theta}
 \Theta(\lambda,\s)=\frac{P_{\s}(\lambda)}{\D P_{\s}(\lambda)}
\end{equation}
is jointly concave on $\{(\lambda,\s):\s\in\K,
                    \lambda\in C_{\D P_{\s}}\}$. Its spectral matrix lift is also
jointly concave.
\end{proposition}
\begin{proof}
Set $R(\lambda,\s,y)=P_{\s}(\lambda+y\one)$. For fixed $\s$, this is
the pullback of $P_{\s}$ under the linear map
\begin{equation}\label{eq:pullback}
 T:\R^{n+1}\longrightarrow\R^n,\qquad
 T(\lambda,y)=\lambda+y\one,\qquad
 [T]=[\,I_n\mid\one\,].
\end{equation}
The matrix has nonnegative entries and no zero row, so Theorem \ref{thm pullback ide} implies that for fixed $\s$, $R(\lambda,\s,y)$ is an ideal G\aa rding polynomial. 

We claim that: For fixed $\s$, $T^{-1}(\CC_{\D^i P_{\s}}) = \CC_{\partial_y^i R_{\s}}.$ 
Proof of the claim: Note that $T^* (\D^i P_{\s})= \partial_y^i R_{\s}$. We just prove the case that $i=0$ and the case for general $i$ can be proved in the same way.
First, we have that:
\begin{equation*}
    T^{-1} (\CC_{P_s}) \subset \{(s,\lambda,y): T^* P_s(\lambda,y)>0\}
\end{equation*}
Take any $(\lambda,y)\in T^{-1}(\CC_{P_s})$. For any $t\ge 0$, $T\big( (\lambda,y)+ t\one \big)= (\lambda+ t \one + (y+t)\one)$. As a result, we have:
\begin{equation*}
    P_{\s}\Big( T\big( (\lambda,y)+ t\one \big) \Big) = P_{\s}\big( \lambda + y \one + 2t \one \big)>0,
\end{equation*}
because $T(\lambda,y)\in \CC_{P_s}$ by assumption. As a result, we have that $(\lambda,y)\in  \CC_{R_{\s}}$. This shows that $T^{-1}(\CC_{ P_{\s}}) \subset \CC_{R_{\s}}.$

Conversely, for any $(\lambda,y)\in \CC_{R_{\s}}$, we have that:
\begin{equation*}
    P_{\s}(T(\lambda,y)+ t \one)= P_{\s}(T(\lambda+ \frac{t}{2}\one, y+ \frac{1}{2}t ))= T^* P_{\s}\big( (\lambda,y) + \frac{t}{2}\one \big) >0.
\end{equation*}
As a result, we have $T(\lambda,y)\in \CC_{P_{\s}}$. This implies that $(\lambda,y)\in T^{-1}\CC_{P_s}$. Then, we have concludes the proof of this claim.

By Proposition \ref{prop concave pol}, we have that $\CC_{\D^{(i)} P}= \CC_{Pol (p^{(i)})}$ is joint convex. Then, we use the fact that $T$ is linear and the above claim to get that $\CC_{\partial_y^i R}$ is convex. Then, we can apply Lemma \ref{lem:quotient} to get that 
\begin{equation*}
    \Theta (\lambda,\s,y)= \frac{P_{\s}(\lambda+ y \one)}{\D P_{\s}(\lambda+ y \one)}
\end{equation*}
is jointly concave. Restricting this function on the slice $\{y=0\}$ concludes the proof of this Proposition.
\end{proof}

\section{A priori estimates for the radial equation }\label{sec:concavity}
For a positive function $\rho\in C^\infty(X)$, set
\begin{equation}\label{eq:scaled-components}
\begin{gathered}
p_{\rho,z}(s)=\rho(z)^d p_z(s/\rho(z)),\qquad
P_{\rho,z}(\lambda)=\rho(z)^dP_z(\lambda/\rho(z)),\\
\mathcal{C}_{P_{\rho,z}}=\rho(z)\mathcal{C}_{P_z},\qquad
\mathcal{C}_{P_{\rho,z}}^{(1)}=\rho(z)\mathcal{C}_{P_z}^{(1)}.
\end{gathered}
\end{equation}
The radial equation is
\begin{equation}\label{eq:radial-rho}
\kappa_z(\Lambda_\varphi)=\rho(z),\qquad
\varphi\in\SH_d(X,\omega_0),
\end{equation}
or, equivalently,
\begin{equation}\label{eq:radial-zero}
P_{\rho,z}(\Lambda_\varphi)=0,\qquad
\Lambda_\varphi\in\partial \mathcal{C}_{P_{\rho,z}}.
\end{equation}

\subsection{The radial subsolution condition}
Recall that a function $\underline\varphi\in C^2(X,\R)$ is a
$\mathcal C$-subsolution of \eqref{eq:radial-rho} if
\begin{equation}\label{eq:subsolution-slice}
\{\lambda\in\Gamma_d:\kappa_z(\lambda)=\rho(z),\quad
\lambda-\Lambda_{\underline\varphi}(z)\in\Gamma_n^+\}
\end{equation}
is bounded for every $z\in X$. We impose
neither $\Lambda_{\underline\varphi}\in\Gamma_d$ nor
$\Lambda_{\underline\varphi}\in \mathcal{C}_{P_z}^{(1)}$.
In particular, $\kappa_z(\Lambda_{\underline\varphi})$ need not be defined.

\begin{proposition}\label{prop:subsolution}
For $\underline\varphi\in C^2(X,\R)$, the following are equivalent:
\begin{enumerate}
\item $\underline\varphi$ is a $\mathcal C$-subsolution of
\eqref{eq:radial-rho}.
\item For every $z\in X$,
\begin{equation}\label{eq:radial-derived}
\rho(z)^{-1}\Lambda_{\underline\varphi}(z)\in \mathcal{C}_{P_z}^{(1)},
\quad\text{equivalently}\quad
\partial_IP_z(\rho(z)^{-1}\Lambda_{\underline\varphi}(z))>0
\quad(1\le |I|\le d-1).
\end{equation}
\item There are $\delta,R>0$ such that, for every $z\in X$,
\begin{equation}\label{eq:81}
\left\{y\ge0:\Lambda_{\underline\varphi}(z)-2\delta\one+y
\in\rho(z)\partial \mathcal{C}_{P_z}\right\}\subset\{|y|\le R\}.
\end{equation}
\end{enumerate}
\end{proposition}
\begin{proof}
Fix $z$ and write $\mu=\Lambda_{\underline\varphi}(z)$, $Q_z=P_{\rho,z}$. From the scaling identity \eqref{eq:scaled-components}, and Proposition~\ref{prop:derived-component}, the condition (ii) is equivalent to $\mu\in C_{Q_z}^{(1)}$, or equivalently \[
\partial_IQ_z(\mu)>0,\quad \forall 1\le|I|\le d-1.
\]
On the other hand, by Lemma~\ref{lem:radial-scale}, the set in \eqref{eq:subsolution-slice} is \[
\rho(z)\partial \mathcal{C}_{P_z}\cap (\mu+\Gamma_n)=\partial \mathcal{C}_{Q_z}\cap (\mu+\Gamma_n).
\]

(i) implies (ii): If (ii) fails at some $z$, by Proposition~\ref{prop:derived-component}, there exists some $I\subset [n]$, $1\le |I|\leq d-1$, such that $\partial_IQ_z(\mu)\leq 0$. Choose $i\in I$, then multi-affinity imples $\partial_i\partial_IQ_z=0$. Consequently, \[
\partial_IQ_z(\mu+te_i)=\partial_IQ_z(\mu)\leq 0.
\]
So the whole ray $x_t=\mu+te_i$ lies outside $\mathcal{C}^{(1)}_{Q_z}$. Hence $x_t\notin \overline{\mathcal{C}_{Q_z}}$. For each $t$, choose $T$ so large that $T\one\in \mathcal{C}_{Q_z}$, and then choose $L>0$ with $\mu+L\one- T\one\in \Gamma_n$. It follows that \[
x_t+L\one=\mu+L\one+te_i\in \mathcal{C}_{Q_z},\quad t\ge 0.
\]
Thus \[
s_*:=\inf\{s\geq 0:x_t+s\one\in \mathcal{C}_{Q_z}\}
\] is finite. $s_*$ is strictly positive since $x_t\notin\overline{C_{Q_z}}$. Hence \[
x_t+s_t\one\in \partial \mathcal{C}_{Q_z},\quad x_t+s_t\one-\mu\in \Gamma_n, \quad |x_t+s_t\one-\mu|\geq t. 
\]
We get a contradiction as $t\to \infty$. 

(ii) implies (iii): By compactness of $X$, there exists a positive $m>0$ such that \[
\partial_IQ_z(\mu)>m, \quad \forall 1\le |I|\le d-1.
\]
By continuity, we can choose $\delta>0$ small, so that \[
\partial_IQ_z(\mu-2s\one)\geq \frac{m}{2} \quad \forall 0<s\le\delta, 1\le|I|\le d-1.
\]
By the multi-affine expansion formula \[
Q_z(\mu-2\delta\one+y)=Q_z(\mu-2\delta\one)+\sum_{1\le|I|\le d}\partial_IQ_{z}(\mu-2\delta\one)\prod_{i\in I}y_i\geq Q_z(\mu-2\delta\one)+\frac{m}{2}\sum_{i}y_i.
\]
So if $y$ belongs to the set in (iii), i.e. $Q_z(\mu-2\delta\one+y)=0$, it follows that \[
|y|\leq \sum_iy_i\leq C, \quad \text{  for a uniform constant } C.
\]

(iii) implies (i): Take $\lambda\in \partial \mathcal{C}_{Q_{z}}\cap (\mu+\Gamma_n)$. Define $y=\lambda-\mu+2\delta\one\in \Gamma_n$. Then $\mu-2\delta\one+y\in \partial \mathcal{C}_{Q_z}$, so $y$ is a point in the bounded set of (iii) and $|y|\leq R$.  It follows that $|\lambda-\mu|\leq |y|\leq R$. So the set in (i) is bounded. 
\end{proof}

\begin{remark}
    These conditions in the above proposition still hold under sufficiently small $C^2$ perturbations
of $\underline\varphi$ and $C^0$ perturbations of $\rho$.
Consequently, a $C^2$ subsolution may be replaced by a smooth one.
The constants $\delta,R$ are uniform for a compact family of positive
functions $\rho$ having a common subsolution.
\end{remark}
\begin{proposition}\label{res:3-1}
If $\underline\varphi$ is a $\mathcal C$-subsolution at $\sigma$,
there is $\varepsilon>0$ such that it is a subsolution of
\eqref{eq:radial-rho} whenever
\begin{equation}\label{eq:197}
0<\rho(z)\le\sigma+\varepsilon\qquad(z\in X).
\end{equation}
For any fixed lower bound $\rho\ge\rho_->0$, the constants in
\eqref{eq:81} can be chosen independently of $\rho$ satisfying
\eqref{eq:197}.
\end{proposition}
\begin{proof}
The strict inequalities in \eqref{eq:radial-derived} hold at
$\sigma+\varepsilon$ for sufficiently small $\varepsilon>0$.
For any strictly right-Noetherian polarization $P$,
\begin{equation}\label{eq:coordinate-rays}
\mu\in \mathcal{C}_P^{(1)}
\quad\Longleftrightarrow\quad
\text{for every }i\text{ there is }R_i>0
\text{ with }\mu+R_ie_i\in \mathcal{C}_P.
\end{equation}
Indeed, the forward implication follows from upward stability and
$P(\mu+Re_i)=P(\mu)+R\partial_iP(\mu)$.
For the reverse implication, let $\varnothing\ne I\subset[n]$ with
$|I|\le d-1$, and choose $i\in I$. Multi-affinity makes $\partial_IP$
independent of $\lambda_i$. Since $\mu+R_ie_i\in \mathcal{C}_P$,
Proposition~\ref{prop:derived-component} gives
\[
 \partial_IP(\mu)=\partial_IP(\mu+R_ie_i)>0.
\]
This holds for every such $I$, so $\mu\in \mathcal{C}_P^{(1)}$.
Thus inclusion of positive components implies inclusion of their
first derived components. Proposition~\ref{prop:radial-nesting} gives
\[
(\sigma+\varepsilon)\mathcal{C}_{P_z}^{(1)}
\subset\rho(z)\mathcal{C}_{P_z}^{(1)}
\qquad(0<\rho(z)\le\sigma+\varepsilon).
\]
Proposition~\ref{prop:subsolution} now applies. Uniformity follows by taking
the minima in its proof over
$X\times[\rho_-,\sigma+\varepsilon]$.
\end{proof}

\subsection{The \texorpdfstring{$C^0$}{C0} estimate}\label{sec:source-5-1}
The following $C^0$-estimate is due to Szekelyhidi \cite{Sze18}, we give the details here for the reader's convenience. 
\begin{lemma}\label{lem:C0-estimate}
Let $\underline\varphi$ be a smooth $\mathcal C$-subsolution of
\eqref{eq:radial-rho}, and normalize a solution by
$\int_X\varphi\chi^n=0$. Then $\|\varphi\|_{C^0}\le C$, where
$C$ depends on the fixed background, $\|\underline\varphi\|_{C^2}$,
and the slice constants $\delta,R$ in \eqref{eq:81}.
\end{lemma}
\begin{proof}
Since $\Lambda_\varphi\in\Gamma_d\subset\Gamma_1$,
\begin{equation}\label{eq:92}
\Delta_\chi\varphi>-\tr_\chi\omega_0\ge-C.
\end{equation}
Put $v=\varphi-\underline\varphi-\sup_X(\varphi-\underline\varphi)$.
Then $\sup_Xv=0$ and $\Delta_\chi v\ge-C$. The Green-function
argument with a Gauduchon representative gives
$\int_X|v|\chi^n\le C$ \cite[proof of Proposition 10]{Sze18}.

Choose a unit coordinate ball centered at a minimum point of $v$, and set
\[
\ell=\min_Xv,\qquad w(z)=v(z)+\epsilon|z|^2,\qquad
E=\{z\in B_1:w(z)\le\ell+\epsilon/2,\ D_\R^2w(z)\ge0\}.
\]
Choose $\epsilon>0$ so that
$\epsilon\ddbar|z|^2\le\delta\chi$, where $\delta$ is as in
Proposition~\ref{prop:subsolution}. For $|p|<\epsilon/4$, a minimum of
$w-p\cdot z$ lies in $E\cap B_1$, so $B_{\epsilon/4}\subset Dw(E)$.
On $E$, eigenvalue monotonicity gives
\[
\ddbar v\ge-\delta\chi,\qquad
y:=\Lambda_\varphi-\Lambda_{\underline\varphi}+2\delta\one
\ge\delta\one.
\]
The eigenvalues are ordered in the same way. Since
$\Lambda_\varphi\in\rho(z)\partial C_{P_z}$, \eqref{eq:81}
gives $|y|\le R$. Thus $|(w_{i\bar j})|\le C$ on $E$, and
the real--complex determinant inequality and the area formula give
\[
c\epsilon^{2n}\le\int_E\det D_\R^2w
\le2^{2n}\int_E\det(w_{i\bar j})^2\le C|E|.
\]
If $\ell<-\epsilon$, then
\[
v\le w\le\ell+\epsilon/2<0\quad\text{on }E,
\qquad
|E|\le\frac{\int_E|v|}{|\ell+\epsilon/2|}
\le\frac C{|\ell+\epsilon/2|}.
\]
Hence $\ell\ge-C$. The oscillation bound and the normalization yield
\begin{equation}\label{eq:93}
\|\varphi\|_{C^0}\le C.
\end{equation}
The proof uses the bounded slices of $\underline\varphi$, without
evaluating $\kappa$ at its eigenvalues.

\end{proof}

\subsection{The \texorpdfstring{$C^2$}{C2} estimate}\label{sec:source-5-2}
 We summarize the method of Hou-Ma-Wu~\cite{HMW10} and the Hermitian extension by Szekelyhidi~\cite{Sze18} on the $C^2$-estimate into the following toolbox. Its proof is identical to the proof of \cite[Section 4, Proposition 13]{Sze18}. 
\begin{lemma}[Szekelyhidi~\cite{Sze18}]
\label{lem:szekelyhidi-differential-estimate}
Let $(X,\chi)$ be a compact Hermitian manifold of dimension $n\ge2$,
let $\beta$ be a fixed smooth real $(1,1)$-form, and let
$u\in C^4(X,\R)$. Put
\[
g=\beta+\ddbar u,\qquad
\lambda=\lambda_\chi(g),\qquad
\lambda_1\ge\cdots\ge\lambda_n.
\]
Assume $\osc_Xu\le L$ and $\tr_\chi g\ge-B$, where $L,B\ge0$.
Suppose a smooth positive Hermitian coefficient tensor is diagonal in
$\chi$-unitary frames diagonalizing $g$, with
\[
F^{i\bar j}=\delta_{ij}F_i,\qquad
0<F_1\le\cdots\le F_n,\qquad \FF=\sum_iF_i.
\]
For constants $C_0,\vartheta>0$, assume the following conditions
in these frames; all derivatives are Chern covariant derivatives.
\begin{enumerate}
\item For each $p$,
\begin{equation}\label{eq:sze-first-derivative}
\left|\sum_iF_i g_{i\bar i,p}\right|\le C_0\FF.
\end{equation}
\item For each $p$,
\begin{equation}\label{eq:sze-second-derivative}
\sum_iF_i g_{i\bar i,p\bar p}
\ge-\sum_{i\ne j}\frac{F_i-F_j}{\lambda_i-\lambda_j}
|g_{i\bar j,p}|^2-C_0\FF.
\end{equation}
The divided differences have continuous nonpositive extensions at
coincident eigenvalues.
\item At every point of $X$, at least one of the alternatives
\begin{equation}\label{eq:sze-subsolution-alternative}
\sum_iF_i(\beta_{i\bar i}-\lambda_i)\ge\vartheta\FF,
\qquad\text{or}\qquad
F_i\ge\vartheta\FF\quad\text{for all }i
\end{equation}
holds.
\end{enumerate}
Then
\begin{equation}\label{eq:sze-hessian-estimate}
\|\ddbar u\|_{C^0(X,\chi)}
\le C\left(1+\sup_X|\nabla u|_\chi^2\right),
\end{equation}
where $C$ depends only on $(X,\chi)$, $\|\beta\|_{C^2}$,
$L,B,C_0$, and $\vartheta$.
\end{lemma}

With this powful tool, we can show the $C^2$-estimate below:
\begin{proposition}\label{res:5-1}
Let $\underline\varphi$ be a smooth $\mathcal C$-subsolution of
\eqref{eq:radial-rho}, and let $\varphi$ be a smooth solution. Then
\begin{equation}\label{eq:96}
\|\ddbar\varphi\|_{C^0(X,\chi)}
\le C\left(1+\sup_X|\nabla\varphi|_\chi^2\right).
\end{equation}
The constant depends on the fixed data, $\underline\varphi$,
$\min_X\rho$, $\|\rho\|_{C^2}$, and the constants in
\eqref{eq:81}. It is uniform for compact smooth families of positive
$\rho$ with a common subsolution.
\end{proposition}

We need only to check the three conditions in the above lemma.

For $\mathbf s=(s_0,\ldots,s_{d-1})$ with
$s_0\ge\cdots\ge s_{d-1}$, write
\begin{equation}\label{eq:149}
\begin{gathered}
p_{\mathbf s}(x)=d!\int_{s_0}^x\int_{s_1}^{t_1}\cdots
\int_{s_{d-1}}^{t_{d-1}}dt_d\cdots dt_1,\qquad
P_{\mathbf s}=\Pol_n(p_{\mathbf s}),\\
D_{\one}P_{\mathbf s}=\sum_i\partial_iP_{\mathbf s}.
\end{gathered}
\end{equation}
A monic right-Noetherian polynomial is determined by its largest
derivative roots, by successive integration. Thus
\begin{equation}\label{eq:scaled-roots}
s_j(z)=\rho(z)r_j(z),\qquad P_{\rho,z}=P_{\mathbf s(z)}.
\end{equation}

\begin{lemma}\label{lem gradient P for}
On $\{(\lambda,\s):\s\in\K,
                    \lambda\in \overline{C_{ P_{\s}}}\}$, we have that:
\begin{equation}\label{eq:smonotone}
 -\partial_{s_j}P_{\s}\ge0\quad(0\le j<d),\qquad
 -\sum_{j=0}^{d-1}\partial_{s_j}P_{\s}=\D P_{\s}.
\end{equation} 
\end{lemma}
\begin{proof}
Differentiating the lower integration limit corresponding to $s_i$ gives
\begin{equation}\label{eq:parameterderivative}
 -\partial_{s_i}p_{\s}(x)=C_i p_i(x),\qquad
 p_i(x)=i!\frp_i(x;s_0,\ldots,s_{i-1}),
\end{equation}
where $p_0=1$ and
\begin{equation}\label{eq:Ci}
 C_i=\frac{d!}{i!}\frp_{d-i-1}(s_i;s_{i+1},\ldots,s_{d-1})\ge0,
 \qquad C_{d-1}=d.
\end{equation}
The nonnegativity follows from the ordering of the integration limits.
For $k<i$, the largest root of $p_i^{(k)}$ is $s_k$, the same as that of
$p_{\s}^{(k)}$; also $p_i^{(i)}=i!>0$. This implies that $p_{\s} \triangleright p_i$. Since polarization preserves domination according to Theorem \ref{UpsilonDominance}, we have that:
\[
 -\partial_{s_i}P_{\s}
       =C_i\Pol_n(p_i)\ge0\quad\hbox{on }\mathcal{C}_{P_{\s}}.
\]
When $C_i=0$, this inequality is the identity $0=0$. For fixed $\s$, continuity extends the inequalities to
$\overline{\mathcal{C}_{P_{\s}}}$. The translation identity below holds
for every $(\lambda,\s)$.

Simultaneous translation of all limits and the integration variable gives
\begin{equation}\label{eq:translation}
 p_{\s+t\one}(x+t)=p_{\s}(x),\qquad
 P_{\s+t\one}(\lambda+t\one)=P_{\s}(\lambda).
\end{equation}
The vectors $\one_d$ and $\one_n$ act on the root parameters and eigenvalues, respectively. Differentiation at $t=0$ yields
$\sum_i\partial_{s_i}P_{\s}+\D P_{\s}=0$.
\end{proof}

\begin{lemma}\label{lem:radial-boundary-derivatives}
On $\{(\lambda,\s):\s\in\K,
                    \lambda\in \overline{C_{ P_{\s}}}, s_0 > s_1 \}$, for every real vector $V\in\R^n\times\R^d$, we have that:
\begin{equation}\label{eq:tangent-concavity}
D_{(\lambda,\mathbf s)}P_{\mathbf s}[V]=0
\quad\Longrightarrow\quad
D^2_{(\lambda,\mathbf s)}P_{\mathbf s}[V,V]\le0.
\end{equation}
\end{lemma}
\begin{proof}
Proposition~\ref{prop concave pol} gives joint
concavity of $P_{\mathbf s}^{1/d}$ on its positive component. At an
interior point, if $DP_{\mathbf s}[V]=0$, then
\[
D^2P_{\mathbf s}[V,V]
=dP_{\mathbf s}^{1-1/d}D^2(P_{\mathbf s}^{1/d})[V,V]\le0.
\]
Approach the given boundary point through positive components with
strictly ordered roots, and replace $V$ at each point by
\[
V-\frac{DP_{\mathbf s}[V]}{D_{\one}P_{\mathbf s}}(\one,0).
\]
These vectors are tangent and converge to $V$. Since
$D_{\one}P_{\mathbf s}>0$ at the limit, continuity proves
\eqref{eq:tangent-concavity}.
\end{proof}

At a solution of \eqref{eq:radial-rho}, put $\lambda=\Lambda_\varphi$
and $\mathbf s=\rho\mathbf r=(\rho(z)r_0(z),...,\rho(z)r_{d-1}(z))$. Define
\begin{equation}\label{eq:normalized-linearization}
F_i:=\frac{\partial_iP_{\mathbf s}(\lambda)}{D_{\one}P_{\mathbf s}(\lambda)},
\qquad F_i>0,\qquad \sum_iF_i=1.
\end{equation}
Write $g=\omega_\varphi$, $A=\chi^{-1}g$, and
$\widehat P(A,\mathbf s)=P_{\mathbf s}(\lambda(A))$. We use the corresponding
matrix coefficients $F^{i\bar j}$: in a $\chi$-unitary frame
diagonalizing $\omega_\varphi$, they are
\[
F^{i\bar j}
=\frac{\partial_{g_{i\bar j}}
\bigl[P_{\mathbf s}(\lambda(\chi^{-1}g))\bigl]}{D_\one P_{\mathbf s}}=F_i\delta_{ij}.
\]
The coefficients $F_i$ can also be expressed in terms of the derivatives of $\kappa_z$ as follows:
$P_{\kappa_z(\lambda)\mathbf r}(\lambda)=\kappa_z(\lambda)^dP_{\mathbf r}(\kappa_z(\lambda)^{-1}\lambda)=0$ since $\kappa_z(\lambda)^{-1}\lambda\in \partial \mathcal{C}_{P_z}$. All eigenvalue and parameter derivatives in the next three identities are taken
with $\lambda$ and $\s$ independent, and are then evaluated at
$\s=\kappa_z(\lambda)\mathbf r$. Differentiating the identity gives
\begin{equation}\label{eq:scale-differentiation}
\partial_iP_{\kappa_z(\lambda)\mathbf r}+\sum_\alpha r_\alpha\partial_{s_\alpha}P_{\kappa_z(\lambda)\mathbf r}\partial_{i}\kappa_z(\lambda)=0.
\end{equation}
Since $P_{\s}(\lambda)$ is jointly homogeneous of degree $d$ in $(\lambda,\s)$, 
\begin{equation}\label{eq:joint-homogeneity}
\sum_{\alpha}\kappa_z(\lambda)r_\alpha\partial_{s_\alpha}P_{\kappa_z(\lambda)\mathbf r}+\sum_{i}\lambda_i\partial_{i}P_{\kappa_z(\lambda)\mathbf r}=dP_{\kappa_z(\lambda)\mathbf r}(\lambda)=0.
\end{equation}
Multiply \eqref{eq:scale-differentiation} by $\kappa_z(\lambda)$ and subtract \eqref{eq:joint-homogeneity} multiplied by $ \partial_i \kappa_z(\lambda)$ to obtain
\begin{equation*}
    \kappa_z (\lambda) \partial_iP_{\kappa_z(\lambda)\mathbf r} = (\sum_{j} \lambda_j \partial_j P_{\kappa_z(\lambda)\mathbf r} )  \partial_i \kappa_z(\lambda)
 \end{equation*}
This implies that
\begin{equation}\label{eq:scale-linearization}
F_i=\frac{\partial_iP_{\kappa_z(\lambda)\mathbf r}}{\sum_j \partial_jP_{\kappa_z(\lambda)\mathbf r}}=\frac{\partial_i\kappa_z(\lambda)}{\sum_j\partial_j\kappa_z(\lambda)}.
\end{equation}

All derivatives below are Chern covariant derivatives. Differentiating
$\widehat P(A,\mathbf s)=0$ gives
\begin{equation}\label{eq:112}
F^{i\bar j}g_{i\bar j,p}+\sum_{j=0}^{d-1}
\frac{\partial_{s_j}P_{\mathbf s}}{D_{\one}P_{\mathbf s}}(s_j)_p=0.
\end{equation}
A second differentiation gives
\begin{equation}\label{eq:114}
\begin{aligned}
F^{i\bar j}g_{i\bar j,p\bar p}+\mathcal J_p+\sum_{j=0}^{d-1}
\frac{\partial_{s_j}P_{\mathbf s}}{D_{\one}P_{\mathbf s}}(s_j)_{p\bar p}=0,
\end{aligned}
\end{equation}
where \[
\mathcal J_p:=\frac{D^2\widehat P[(A_{,p},\mathbf s_p),
(A_{,\bar p},\mathbf s_{\bar p})]}{D_{\one}P_{\mathbf s}}.
\]
The first and second order derivatives of $\mathbf s=\rho\mathbf r$ are bounded in terms
of the fixed data and $\|\rho\|_{C^2}$. Hence, we can use Lemma \ref{lem gradient P for} to get that:
\begin{equation}\label{eq:radial-differentiated-bounds}
\left|F^{i\bar j}g_{i\bar j,p}\right|\le C,\qquad
F^{i\bar j}g_{i\bar j,p\bar p}\ge-\mathcal J_p-C.
\end{equation}

Using (\ref{eq:112}), the vector
$V_p=(g_{1\bar1,p},\ldots,g_{n\bar n,p},\mathbf s_p)$
satisfies $DP_{\mathbf s}[V_p]=0$. Applying
\eqref{eq:tangent-concavity} to its real and imaginary parts and
using the spectral second-derivative formula
\cite[\S4, equations (65)--(67)]{Sze18}, we obtain
\begin{equation}\label{eq:117}
\begin{aligned}
(D_{\one}P_{\mathbf s})\mathcal J_p
&=\sum_{i,j}(\partial^2_{ij}P_{\mathbf s})g_{i\bar i,p}g_{j\bar j,\bar p}+2\operatorname{Re}\sum_{i,\alpha}(\partial_{i}\partial_{s_\alpha}P_{\mathbf s})g_{i\bar i,p}s_{\alpha,\bar p}+\sum_{\alpha,\beta}(\partial^2_{s_{\alpha}s_{\beta}}P_{\mathbf s})s_{\alpha,p}s_{\beta,\bar p}
+(D_{\one}P_{\mathbf s})\sum_{i\ne j}\frac{F_i-F_j}{\lambda_i-\lambda_j}|g_{i\bar j,p}|^2\\
&\le (D_{\one}P_{\mathbf s})\sum_{i\ne j}\frac{F_i-F_j}{\lambda_i-\lambda_j}|g_{i\bar j,p}|^2\le0.
\end{aligned}
\end{equation}
The last nonpositivity follows from symmetry and concavity of $P_{z}^{\frac{1}{d}}$. Those formulas make sense even when the eigenvalues are not distinct, see the discussion in \cite[Section 4]{Sze18}.

\begin{lemma}\label{res:5-2}
Under the hypotheses of Proposition~\ref{res:5-1}, put
$\underline A=\chi^{-1}(\omega_0+\ddbar\underline\varphi)$,
and let $\underline\lambda_1\ge\cdots\ge\underline\lambda_n$
be its eigenvalues. There is $\vartheta>0$, depending on $(\delta,R)$ from Proposition~\ref{prop:subsolution} and the fixed data on the manifold, such that, for every $z\in X$
and every ordered $\lambda$ with $\kappa_z(\lambda)=\rho(z)$,
at least one of the following holds:
\begin{equation}\label{eq:97}
\sum_iF_i(\underline\lambda_i-\lambda_i)>\vartheta;
\end{equation}
\begin{equation}\label{eq:98}
F_i>\vartheta\qquad(1\le i\le n).
\end{equation}
Here $F_i$ is given by
\eqref{eq:normalized-linearization}.
\end{lemma}
\begin{proof}
Fix $z$, write $\mu=\underline\lambda(z)$, and choose uniform
$\delta,R>0$ as in Proposition~\ref{prop:subsolution}, with
\begin{equation}\label{eq:101}
\{y\ge0:\mu-2\delta\one+y\in\rho(z)\partial C_{P_z}\}
\subset\{|y|\le R\}.
\end{equation}
By Proposition~\ref{prop:subsolution}(ii) and compactness, we also have $\mu-2\delta\one\in \mathcal{C}_{P_{\rho,z}}^{(1)}.$

Set $N=\nabla\kappa_z(\lambda)/|\nabla\kappa_z(\lambda)|$. Using Lemma~\ref{lem gradient P for}, we have that $N_i>0$ for any $i$. This implies that:
\begin{equation}\label{e n1 est}
1\le N\cdot\one\le\sqrt n.
\end{equation}
According to \eqref{eq:scale-linearization}, we have that 
\begin{equation}\label{e fini}
    F_i=\frac{N_i}{N\cdot\one}.
\end{equation}
If $N\cdot(\mu-\lambda)\ge\delta$, then
$\sum_iF_i(\mu_i-\lambda_i)\ge\delta/\sqrt n$.
Otherwise, put
\begin{equation}\label{e l lower}
    \ell=N\cdot(\lambda-\mu+2\delta\one)>\delta,
\qquad x_i=\mu-2\delta\one+\frac\ell{N_i}e_i.
\end{equation}
Since $\mathcal{C}^{(1)}_{P_{\rho,z}}+\overline{\Gamma}_n\subset \mathcal{C}^{(1)}_{P_{\rho,z}}$,  one gets $x_i\in \mathcal{C}_{P_{\rho,z}}^{(1)}$, and
$N\cdot(x_i-\lambda)=0$. If $P_{\rho,z}(x_i)>0$, then
\[
x_i\in \mathcal{C}_{P_{\rho,z}},\qquad
\rho(z)<\kappa_z(x_i)
\le\kappa_z(\lambda)+\nabla\kappa_z(\lambda)\cdot(x_i-\lambda)
=\rho(z),
\]
a contradiction. Thus $P_{\rho,z}(x_i)\le0$.
Strict increase along the diagonal gives $t_i\ge0$ with
$x_i+t_i\one\in\rho(z)\partial C_{P_z}$.
Applying \eqref{eq:101} to $y=\ell N_i^{-1}e_i+t_i\one$ yields
\[
\ell/N_i\le R.
\]
Combining this formula with (\ref{e n1 est}), (\ref{e fini}) and (\ref{e l lower}), we can get that:
$$
F_i= \frac{N_i}{N \cdot \one} \ge \frac{\ell}{R \sqrt{n}} >\frac\delta{R\sqrt n}.
$$
The alternatives follow with
$\vartheta=(2\sqrt n)^{-1}\min\{\delta,\delta/R\}$.
\end{proof}

\begin{remark}\label{res:5-3}
In a $\chi$-unitary frame diagonalizing $A$, order its eigenvalues
by $\lambda_1\ge\cdots\ge\lambda_n$. Symmetry and concavity of
$P_z^{\frac{1}{d}}$ give
\begin{equation}\label{eq:99}
F_1\le\cdots\le F_n.
\end{equation}
By the Schur--Horn theorem and the rearrangement inequality,
\[
\sum_iF_i\underline A_{i\bar i}\ge\sum_iF_i\underline\lambda_i.
\]
 In fact, the Schur--Horn theorem says that the vector $(\underline A_{1\bar{1}},\cdots,\underline A_{n\bar{n}})$ is contained in the convex hull of $(\underline\lambda_{\sigma(1)},\cdots,\underline\lambda_{\sigma(n)})$,  where $\sigma$ is a permutation of the indices.  Moreover,  since $F_i$ is in the increasing order,  $F_i \underline\lambda_i$ will be minimized if $\underline\lambda_i$ is in the decreasing order,  hence the inequality above. 

Consequently, either \eqref{eq:98} holds or
\begin{equation}\label{eq:100}
\sum_iF_i(\underline A_{i\bar i}-\lambda_i)>\vartheta.
\end{equation}
\end{remark}
\begin{proof}[Proof of Proposition~\ref{res:5-1}]
Apply Lemma~\ref{lem:szekelyhidi-differential-estimate} with
\[
u=\varphi-\underline\varphi,\qquad
\beta=\omega_0+\ddbar\underline\varphi,\qquad
g=\omega_\varphi.
\]
Lemma~\ref{lem:C0-estimate} bounds $\osc_Xu$, and
$\Lambda_\varphi\in\Gamma_d\subset\Gamma_1$ gives
$\tr_\chi g>0$. The normalized coefficients
\eqref{eq:normalized-linearization} satisfy $\FF=1$ and
$F_1\le\cdots\le F_n$ by \eqref{eq:99}.
Equations \eqref{eq:radial-differentiated-bounds} and \eqref{eq:117}
give conditions (i)--(ii) of the lemma. Lemma~\ref{res:5-2} and
Remark~\ref{res:5-3} give condition (iii).
Thus \eqref{eq:sze-hessian-estimate} holds for $u$.
Since $\underline\varphi$ is fixed, this is equivalent to the asserted
estimate \eqref{eq:96} for $\varphi$. The constants are uniform under
the stated bounds on the data and the subsolution constants.
\end{proof}

\subsection{The gradient estimate}\label{sec:source-5-3}
We apply the blow-up method of Dinew--Ko\l odziej~\cite{DK17}, 
in the viscosity formulation of \cite[Section 6]{Sze18} to show the gradient estimate.

\begin{lemma}\label{res:5-5}
Under the hypotheses of Proposition~\ref{res:5-1}, a smooth solution
of \eqref{eq:radial-rho} satisfies
\begin{equation}\label{eq:148}
\|\nabla\varphi\|_{C^0(X,\chi)}\le C.
\end{equation}
The bound is uniform for positive functions $\rho$ with a uniform positive lower bound and $C^2$-bound.
\end{lemma}
\begin{proof}
Suppose that normalized solutions $\varphi_j$ to (\ref{eq:radial-rho}) with right-hand-side $\rho_j$
satisfy
\[
K_j=\|\nabla\varphi_j\|_{C^0(\chi)}\longrightarrow\infty,
\qquad |\nabla\varphi_j(p_j)|_\chi=K_j,
\qquad p_j\longrightarrow p\in X.
\]
Choose holomorphic charts $\Psi_j:B_{2r}\to X$ with
$\Psi_j(0)=p_j$ and $(\Psi_j^*\chi)(0)=I$, whose coefficients
and inverses have uniform bounds. On $B_{rK_j}\subset\C^n$, put
\[
\begin{gathered}
\Phi_j(\zeta)=\Psi_j(\zeta/K_j),\qquad
v_j=\varphi_j\circ\Phi_j-\varphi_j(p_j),\\
g_j=K_j^2\Phi_j^*\chi,\qquad a_j=\Phi_j^*\omega_0.
\end{gathered}
\]
Then $g_j\to I$ and $a_j\to0$ locally uniformly.
The $C^0$ estimate and Proposition~\ref{res:5-1} give
\[
|v_j|+|\nabla v_j|+\max_{p,q}|(v_j)_{p\bar q}|\le C,
\qquad |\nabla v_j(0)|=1.
\]
In particular, $|\Delta_{\R^{2n}}v_j|\le C$. Interior
$W^{2,q}$ estimates, $q>2n$, and compact Sobolev embedding give,
after passing to a subsequence,
\[
v_j\longrightarrow v\quad\text{in }C^1_{\mathrm{loc}}(\C^n),
\qquad |v|+|\nabla v|\le C,
\qquad |\nabla v(0)|=1.
\]

Write $\lambda_g$ for eigenvalues relative to a metric $g$.
Homogeneity of $\kappa_z$ gives the rescaled equation
\begin{equation}\label{eq:radial-blow-up}
\begin{gathered}
\lambda_{g_j}(a_j+\ddbar v_j)\in\Gamma_d,\\
\kappa_{\Phi_j(\zeta)}
\bigl(\lambda_{g_j}(a_j+\ddbar v_j)\bigr)
=K_j^{-2}\rho_j(\Phi_j(\zeta))\longrightarrow0.
\end{gathered}
\end{equation}
Let $h\in C^2$ touch $v$ from above or below at $\zeta_0$.
Adding or subtracting $|\zeta-\zeta_0|^4$ makes the contact strict.
There are corresponding local extrema $\zeta_j\to\zeta_0$ of
$v_j-h$, and
\[
\eta_j=\lambda_{g_j}(a_j+\ddbar h)(\zeta_j)
\longrightarrow\xi=\lambda\bigl((h_{p\bar q})(\zeta_0)\bigr).
\]
For an upper test, $\ddbar h\ge\ddbar v_j$ at $\zeta_j$.
Eigenvalue monotonicity and upward stability of $\Gamma_d$ give
$\eta_j\in\Gamma_d$, hence $\xi\in\overline{\Gamma_d}$.

For a lower test, suppose $\xi\in\Gamma_d$. Then
$\eta_j\in\Gamma_d$ for large $j$ and
$\ddbar h\le\ddbar v_j$ at $\zeta_j$. Monotonicity gives
\[
0<\kappa_{\Phi_j(\zeta_j)}(\eta_j)
\le K_j^{-2}\rho_j(\Phi_j(\zeta_j))\longrightarrow0,
\]
whereas the left-hand side converges to $\kappa_p(\xi)>0$.
Thus every lower test has $\xi\notin\Gamma_d$.

These inequalities make $v$ a $\Gamma_d$-solution in the sense of
\cite[Definition 15]{Sze18}. The cone $\Gamma_d$ is open, convex,
symmetric, contains $\Gamma_n$, and lies in $\Gamma_1$.
The Liouville theorem \cite[Theorem 20]{Sze18} therefore makes the
bounded Lipschitz function $v$ constant, contradicting
$|\nabla v(0)|=1$.
\end{proof}

\subsection{Higher order estimate}
\begin{lemma}\label{lem higher order}
Under the hypotheses of Proposition~\ref{res:5-1}, a smooth solution
of \eqref{eq:radial-rho} satisfies that: for any $k$, there exists a constant $C_k$ such that 
\begin{equation}\label{eq:148}
\|\nabla^k \varphi\|_{C^0(X,\chi)}\le C_k.
\end{equation}
The bound is uniform for positive functions $\rho$ with a uniform positive lower bound and $C^{\max\{2,k-1\}}$-bound.
\end{lemma}
\begin{proof}
    Combining Proposition \ref{res:5-1} and Lemma \ref{res:5-5}, we can get that there exists a constant $B$ such that $|\lambda[\Lambda_{\varphi}]|\le B$. By assumption, there exist constants $m,M$ such that $0<m\le \rho \le M$. Consider 
    \begin{equation*}
        \Sigma= \{(z,\rho,\lambda): z\in X, \rho \in [m,M], |\lambda|\le B, P_{\rho,z}(\lambda)=0, \partial_I P_{\rho,z}(\lambda)\ge 0 (1\le |I|\le d-1) \}.
    \end{equation*}
    This is compact. Lemma \ref{lem:component-closure} implies that for any $(z,\rho,\lambda)\in \Sigma$, $\lambda \in C_{P_{\rho,z}}$. Proposition \ref{prop:radial-nesting} makes sure that the above $\lambda$ lies in the interior of $\Gamma_d$. Lemma \ref{lem:radial-scale} implies that $P_{\rho,z}$ is uniformly elliptic on this compact set. Then, we can apply complex Evans-Krylov \cite{TWWY15} and Schauder estimates to get higher order estimate for $\varphi$.
\end{proof}

\section{Proof of Theorem~\texorpdfstring{\ref{thm:radial}}{1}}\label{sec:continuity}

We first record the linear statement used for openness.
\begin{lemma}\label{res:4-2}
Let $L$ be a smooth uniformly elliptic second-order operator on $X$
with no zeroth-order term. For $k\ge0$ and $0<\gamma<1$, the map
\begin{equation}\label{eq:88}
C^{k+2,\gamma}_0(X)\times\R\longrightarrow C^{k,\gamma}(X),
\qquad(v,a)\longmapsto Lv-a
\end{equation}
is an isomorphism. The subscript $0$ denotes zero $\chi^n$-mean.
\end{lemma}
\begin{proof}
Deformation to the Hermitian Laplacian does not change the index. Let $\chi_0$ be a Gauduchon metric, we show that $C^{k+2,\gamma}_0(X)\times\R\longrightarrow C^{k,\gamma}(X),
(v,a)\longmapsto L_0v-a$ has index zero, where $L_0=\chi_0^{i\bar j}\partial_i\partial_{\bar j}$. The map $L_0:C^{k+2,\gamma}(X)\to C^{k,\gamma}(X)$ has index zero: By maximum principle, $\operatorname{ker}(L_0)=\R$. And the adjoint $L_0^*$ satisfies $(L_0^*w)\chi_0^n=n\sqrt{-1}\partial\bar\partial(w\chi_0^{n-1})$. It has the same principal part as $L_0$ and vanishing 0-th order since $\partial\bar\partial \chi_0^{n-1}=0$. Maximum principle applied again gives $\operatorname{ker}(L_0^*)=\R$. This proves the zero index. Therefore, restriction to zero-mean subspace make the index $-1$. But adding $\R$ makes the map $C^{k+2,\gamma}_0(X)\times\R\longrightarrow C^{k,\gamma}(X),
(v,a)\longmapsto L_0v-a $ have index $0$ again.   If $Lv=a$, evaluation at a maximum and a
minimum gives $a=0$. The strong maximum principle then gives
$v=0$, so the map is injective and hence isomorphic. 
\end{proof}

\begin{proof}[Proof of Theorem~\ref{thm:radial}]
If $(\varphi,c)$ is an elliptic solution, then
\[
\Lambda_\varphi/c\in\partial \mathcal{C}_{P_z}\subset\Gamma_d,
\qquad\kappa_z(\Lambda_\varphi)=c.
\]
Lemma~\ref{res:7-4} gives $c=\sigma$, and
Proposition~\ref{prop:subsolution} shows that $\varphi$ itself is a
$\mathcal C$-subsolution at $\sigma$.

Conversely, let $\underline\varphi$ be a subsolution at $\sigma$, and we may take it smooth.
Choose $\varepsilon>0$ as in Proposition~\ref{res:3-1}.
The definition of $\sigma$ supplies $v\in\SH_d(X,\omega_0)$
such that, setting $h(z)=\kappa_z(\Lambda_v)$,
\[
0<h^-:=\min_Xh\le h(z)\le h^+:=\max_Xh<\sigma+\varepsilon.
\]
Normalize $v$ by $\int_Xv\chi^n=0$. Consider the radial path
\begin{equation}\label{eq:199}
\kappa_z(\Lambda_{\varphi_t})=c_t h(z)^{1-t},\qquad
\varphi_t\in\SH_d(X,\omega_0),\qquad
\int_X\varphi_t\chi^n=0,\qquad
c_t>0,\quad 0\le t\le1.
\end{equation}
At $t=0$, $(\varphi_0,c_0)=(v,1)$.

\emph{Openness.} Write $c_t=e^{a_t}$ and linearize
\[
\log\kappa_z(\Lambda_\varphi)-(1-t)\log h-a=0.
\]
At an admissible solution, the linearization in $(\varphi,a)$ is
$(w,b)\mapsto Lw-b$, where $L$ is uniformly elliptic and has
no zeroth-order term. Lemma~\ref{res:4-2} and the
implicit-function theorem give openness.

\emph{Uniform estimates and closedness.} At a maximum and a minimum
of $\varphi_t-v$, eigenvalue monotonicity gives, respectively,
\[
c_t\le(h^+)^t,\qquad c_t\ge(h^-)^t.
\]
Consequently, with $\rho_t(z)=c_t h(z)^{1-t}$,
\begin{equation}\label{eq:path-scale-bounds}
h^-\le\rho_t(z)\le h^+<\sigma+\varepsilon,
\qquad
\|\rho_t\|_{C^k(X)}\le C_k\quad(k\ge0).
\end{equation}
Proposition~\ref{res:3-1} makes $\underline\varphi$ a common
subsolution for all these radial equations, with uniform slice bounds.
Therefore we can use Lemma \ref{lem higher order} to get uniform $C^k$ estimates for
$\varphi_t$. The constants $c_t$ stay in a compact subinterval
of $(0,\infty)$. The solvable set is therefore closed
in $[0,1]$. At $t=1$ we obtain
$\kappa_z(\Lambda_{\varphi_1})\equiv c_1$, and the necessity
argument gives $c_1=\sigma$.

\emph{Uniqueness.} Suppose $\varphi$ and $v$ are two solutions.
Both have radial constant $\sigma$. The cone $\Gamma_d$ is
symmetric and convex, so its spectral matrix domain is convex
\cite{Bha97}. Along $\eta_s=(1-s)\omega_\varphi+s\omega_v$,
\[
0=\kappa_z(\Lambda_v)-\kappa_z(\Lambda_\varphi)
=\left(\int_0^1\kappa_z^{i\bar j}(\eta_s)\,ds\right)
(v-\varphi)_{i\bar j}.
\]
Here $\kappa_z^{i\bar j}$ denotes differentiation of the spectral
function with respect to the form. By Lemma \ref{lem:radial-scale}, $\kappa_z$ is strictly increasing in each coordinate. This implies that $\kappa_z^{i\bar j}$ is strictly positive definition. As a result, the averaged matrix is uniformly
positive definite on $X\times[0,1]$. The strong maximum principle
implies that $v-\varphi$ is constant.

Finally, testing the two slopes with the solution and applying
Lemma~\ref{res:7-4} proves \eqref{eq:23}. 
\end{proof}

\section{Hermitian complex $k$-Hessian equation and complex quotient equation}\label{sec:source-9}
\subsection{complex $k$-Hessian equation}\label{sec:source-9-1}
Take $d=k$ and $p_z=s^k-e^{\psi(z)}$, with $\psi\in C^\infty(X)$. The radial admissible class is $\SH_k(X,\omega_0)=\{v\in C^\infty(X):\lambda[\chi^{i\bar{j}}(\omega_0+ dd^c v)_{l \bar{j}}]\in \Gamma_k \}$.

\begin{corollary}\label{res:9-1}
If $\SH_k(X,\omega_0)\ne\varnothing$, there are a smooth $\varphi\in\SH_k(X,\omega_0)$ and a unique constant $b\in\R$ such that
\begin{equation}\omega_\varphi^k \wedge \chi^{n-k}=e^{\psi+b}\chi^n.\label{eq:209}\end{equation}
The solution is unique modulo constants, and
\begin{equation}e^b=\sigma^k=\inf_{v\in\SH_k(X,\omega_0)}\max_X e^{-\psi}\frac{\omega_v^n}{\chi^n}=\sup_{v\in\SH_k(X,\omega_0)}\min_X e^{-\psi}\frac{\omega_v^n}{\chi^n}.\label{eq:210}\end{equation}
No K\"ahler assumption is required.
\end{corollary}
\begin{proof}
Here $a_k=1$, $a_i=0$ for $1 \le i \le k-1$ and $a_0=e^{\psi(z)}$. Here $r_0=e^{\psi/k}>0=r_1=\cdots=r_{k-1}$. All these functions are smooth functions on $X$. $a_{k-1}\ge 0$.  Hence Assumption \ref{ass1} holds in this case. 
Note that $P_z(\lambda)= S_k(\lambda)-e^{\psi(z)}$ and $P_z \big(\kappa_z(\lambda)^{-1} \lambda \big)=0.$ This implies that 
\begin{equation*}
    \kappa_z(\lambda)= (\frac{S_k(\lambda)}{e^{\psi(z)}})^{\frac{1}{k}}.
\end{equation*}
Then, we have that:
\begin{equation}\label{eq:23}
    \sigma= \inf_{v\in \SH_k(X,\omega_0)} \max_X (\frac{S_k (\lambda[\Lambda_v])}{e^{\psi(z)}})^{\frac{1}{k}}.
\end{equation}

Since any $v\in \SH_k(X,\omega_0)$ satisfies Definition~\ref{def:radial-subsolution}. Theorem~\ref{thm:radial} gives $\omega_\varphi^k \wedge \chi^{n-k}=\sigma^k e^\psi\chi^n$; set $b=k\log\sigma$. Raising \eqref{eq:23} to the $k$-th power gives \eqref{eq:210}.
\end{proof}
The first min--max formula in \eqref{eq:210} is the Hermitian normalizing constant in \cite[Corollary 1.1]{GS24}. It is a consequence of the radial solvability theorem.

\subsection{complex quotient equation}\label{sec:source-9-2}
Take $d=k$ and $p_z(s)=s^k- e^{\psi} s^l$, with $\psi \in C^{\infty}(X)$. The radial admissible class is $$\SH_k(X,\omega_0)=\{v\in C^\infty(X):\lambda[\chi^{i\bar{j}}(\omega_0+ dd^c v)_{l \bar{j}}]\in \Gamma_k \}$$.

\begin{corollary}[The slope-normalized $J$-criterion]\label{res:9-2}
One has $0<\xi<\infty$. There exist a smooth $\varphi$ and a constant $\gamma>0$ solving
\begin{equation}
\omega_\varphi^{k}\wedge\chi^{n-k}= e^{\psi+b} \omega_\varphi^l \wedge \chi^{n-l},\qquad \varphi\in \SH_k(X,\omega_0),\label{eq:211}
\end{equation}
if and only if some $\underline\varphi\in C^2(X)$ satisfies
\begin{equation}
k \omega_{\varphi}^{k-1}\wedge \chi^{n-k} - e^{\psi+b} l \omega_{\varphi}^{l-1}\wedge \chi^{n-l}>0.
\label{eq:212}\end{equation}
$\varphi$ is unique modulo constants. Here $b$ is given by:
\begin{equation*}
    e^b = \inf_{v\in \SH_k(X,\omega_0)} \max_X e^{-\psi}\frac{S_k(\lambda[\Lambda_v])}{S_l(\lambda[\Lambda_v])}.
\end{equation*}
\end{corollary}
\begin{proof}
First, we have that $a_l=e^{\psi}\ge 0$ and $a_i=0$ for $i \neq l$. We can compute that:
\begin{equation*}
    r_i= (\frac{l! (k-i)! e^{\psi}}{(l-i)! k!})^{\frac{1}{k-l}}, \text{ for } i \le l,
\end{equation*}
and $r_i=0$ for $i \ge l+1$. Hence Assumption 1.1 holds in this case. Note that $P_z(\lambda)= S_k(\lambda)- e^{\psi} S_l(\lambda)$ and $P_z(\kappa_z(\lambda)^{-1}\lambda)=0.$ We can get that:
\begin{equation}
    \kappa_z = (e^{-\psi}\frac{S_k}{S_l})^{\frac{1}{k-l}}.
\end{equation}
This implies that:
\begin{equation*}
    \sigma= \inf_{v\in \SH_k(X,\omega_0)} \max_X (e^{-\psi}\frac{S_k(\lambda[\Lambda_v])}{S_l(\lambda[\Lambda_v])})^{\frac{1}{k-l}}.
\end{equation*}
According to Theorem \ref{thm:radial}, we have that $$e^b=\sigma^{k-l}= \inf_{v\in \SH_k(X,\omega_0)} \max_X e^{-\psi}\frac{S_k(\lambda[\Lambda_v])}{S_l(\lambda[\Lambda_v])}.$$ This Corollary follows immediately from Theorem \ref{thm:radial}.
\end{proof}
Corollary~\ref{res:9-2} recovers \cite[Theorem 1.3, equations (1.19)--(1.21)]{GS24}. 

\section{K\"ahler case}

\begin{lemma}\label{lem kappa est}
    Given $\mu \in \mathbb{R}^n$ and $\tau \ge 0$. Let $p=s^d- \sum_{k=0}^{d-1} a_k s^k$ be a strictly right-Noetherian polynomial with $a_{d-1}\ge 0$. Let $\kappa$ be its radial slope. Suppose that $|\mu_i|\le M$ for every $i$. Then, we have that:
    \begin{equation*}
        \frac{\tau-M}{r(p)} \le \kappa(\mu+ \tau \one) \le \frac{\tau+M}{r(p)},
    \end{equation*}
    for $\tau>M$.
\end{lemma}
\begin{proof}
    Since $|\mu_i|\le M$ for every $i$, we have that:
    \begin{equation*}
        (\tau-M)\one \le \mu+ \tau \one \le (\tau+M) \one.
    \end{equation*}
     By the monotonicity established by Lemma \ref{lem:radial-scale}, we have that:
     \begin{equation}\label{e kappa com}
         \kappa \big( (\tau-M)\one \big) \le \kappa (\mu+ \tau \one) \le \kappa \big( (\tau+M)\one \big).
     \end{equation}
     Next, we compute $ \kappa \big( (\tau-M)\one \big)$ and $\kappa \big( (\tau+M)\one \big)$. Because $P(t \one)=p(t)$,  we have that $r(p) \one \in \partial C_P$. This implies that:
     \begin{equation}\label{e kappa dia}
         \kappa \big( (\tau-M)\one \big) =\frac{\tau-M}{r(p)},\,\,\, \kappa \big( (\tau+M)\one \big) =\frac{\tau+M}{r(p)}.
     \end{equation}
     Combining (\ref{e kappa com}) and (\ref{e kappa dia}), we can conclude the proof of this Lemma.
\end{proof}

\begin{proof}[Proof of Theorem~\ref{main thm 2}]
\emph{Step 1: canonical centering and the subsolution condition.}
Define
\[
\widehat\omega_0:=\omega_0-\gamma\chi,\qquad
q_z(s):=p_z(s+\gamma),\qquad Q_z:=\Pol_n(q_z).
\]
Since $\gamma$ is constant and $\chi$ is closed,
$\widehat\omega_0$ is closed. The binomial expansion gives
\begin{equation}\label{eq:kahler-centered-coefficients}
\begin{gathered}
q_z(s)=s^d-\sum_{j=0}^{d-1}b_j(z)s^j,\qquad b_{d-1}=0,\\
b_j(z)=\sum_{k=j}^{d-1}\binom kj a_k(z)\gamma^{k-j}
-\binom dj\gamma^{d-j}\quad(0\le j\le d-1).
\end{gathered}
\end{equation}
The coefficients $b_j$ are constant for $1\le j\le d-1$.
The largest derivative roots of $q_z$ are
\[
\widehat r_i=r_i-\gamma,\qquad
\widehat r_0(z)>\widehat r_1\ge\cdots
\ge\widehat r_{d-1}=0.
\]
The lower roots are constant. Since
$p_z'(r_0(z))>0$, the implicit-function theorem makes
$r_0(z)$, and hence $\widehat r_0(z)$, smooth.
Thus the data satisfy Assumption~\ref{ass1}.

 Normalized polarization gives
\begin{equation}\label{eq:kahler-centered-polarization}
Q_z(\mu)=P_z(\mu+\gamma\one),\qquad
\widehat\Lambda_v:=\lambda_\chi(\widehat\omega_0+\ddbar v)
=\Lambda_v-\gamma\one.
\end{equation}
Consequently
\begin{equation}\label{eq:kahler-translated-components}
\mathcal{C}_{Q_z}=\mathcal{C}_{P_z}-\gamma\one,\qquad
\mathcal{C}_{Q_z}^{(1)}=\mathcal{C}_{P_z}^{(1)}-\gamma\one.
\end{equation}
Compatibility is preserved, since
$Q_z(\widehat\Lambda_v)=P_z(\Lambda_v)$.
By Lemma~\ref{lem:component-closure} and
Proposition~\ref{prop:radial-nesting},
\begin{equation}\label{eq:kahler-centered-zero-branch}
\partial \mathcal{C}_{Q_z}
=\{\mu\in C_{Q_z}^{(1)}:Q_z(\mu)=0\}
\subset\Gamma_d.
\end{equation}
This proves 
\begin{equation}\label{eq:kahler-canonical-admissibility}
\Lambda_\varphi(z)-\gamma \one \in\Gamma_d\quad(z\in X).
\end{equation}
for every
elliptic zero of $P_z$, independently of existence.

The finitely many strict mixed-derivative inequalities in
Proposition~\ref{prop:derived-component} persist under sufficiently
small $C^2$ perturbations. We may consequently replace a $C^2$
reference in (ii) by a smooth one. Fix such a reference and write
\[
\mu(z):=\widehat\Lambda_{\underline\varphi}(z)\in \mathcal{C}_{Q_z}^{(1)}.
\]
We do not assume that $\mu(z)$ belongs to $\Gamma_d$.

\emph{Step 2: a cohomologically compatible family.}
Let $V:=\int_X\chi^n$. For $\tau\ge0$, put
\begin{equation}\label{eq:kahler-compatible-family}
\begin{gathered}
\omega_{0,\tau}:=\widehat\omega_0+\tau\chi,\qquad
\Lambda_{v,\tau}:=\lambda_\chi(\omega_{0,\tau}+\ddbar v),\qquad
\mu_\tau:=\mu+\tau\one,\\
\ell(\tau):=V^{-1}\int_XQ_z(\mu_\tau)\chi^n,\qquad
q_{\tau,z}:=q_z-\ell(\tau),\qquad
Q_{\tau,z}:=Q_z-\ell(\tau).
\end{gathered}
\end{equation}
Here $\tau$ is a prescribed continuation parameter.
For a closed real $(1,1)$-form $\beta$,
\[
(\beta+\ddbar v)^k-\beta^k
=\ddbar\left(v\sum_{j=0}^{k-1}
(\beta+\ddbar v)^j\wedge\beta^{k-1-j}\right).
\]
Stokes' theorem, and the constancy
of the positive-order coefficients therefore give
\begin{equation}\label{eq:kahler-path-compatibility}
\int_XQ_{\tau,z}(\Lambda_{v,\tau})\chi^n=0
\qquad(v\in C^\infty(X,\R)).
\end{equation}
In particular, \eqref{eq:kahler-fixed-compatibility} implies
$\ell(0)=0$. Upward stability of $\mathcal{C}_{Q_z}^{(1)}$ gives $\mu_\tau\in \mathcal{C}_{Q_z}^{(1)}$,
and
\[
\ell'(\tau)=V^{-1}\int_X\D Q_z(\mu_\tau)\chi^n>0.
\]
Thus $\ell(\tau)> 0$ for $\tau >0$. Only the constant coefficient changes;
the derived component and all lower derivative roots remain fixed.
Moreover,
\[
q_{\tau,z}(\widehat r_1)
=q_z(\widehat r_1)-\ell(\tau)<0.
\]
Since $q_z'>0$ above $\widehat r_1$, the largest real root
$\widehat r_{0,\tau}(z)$ of $q_{\tau,z}$ is simple, smooth, and strictly larger
than $\widehat r_1$. Hence Assumption~\ref{ass1} holds along
the entire path.

\emph{Step 3: large-parameter existence and normalization.}
Monicity gives
\[
\ell(\tau)=\tau^d+O(\tau^{d-1}),\qquad
q_{\tau,z}(\tau(1\pm\eta))
=\tau^d((1\pm\eta)^d-1)+O(\tau^{d-1})
\]
uniformly in $z$, for each fixed $0<\eta<1$.
The opposite signs and positivity of the derivative above
$\widehat r_1$ imply
\begin{equation}\label{e asy r0tau}
\widehat r_{0,\tau}(z)/\tau\longrightarrow1
\quad\text{uniformly on }X.
\end{equation}
For large $\tau$, the reference form
$\omega_{0,\tau}+\ddbar\underline\varphi$ is positive,
so $\SH_d(X,\omega_{0,\tau})\ne\varnothing$.
Let $\kappa_{\tau,z}$ and $\sigma_\tau$ denote the radial scale
and slope for these data. If $|\mu_i|\le M$, then we can use Lemma \ref{lem kappa est} to get that:
\begin{equation}\label{e kap es}
\frac{\tau-M}{\widehat r_{0,\tau}(z)}
\le\kappa_{\tau,z}(\mu_\tau(z))
\le\frac{\tau+M}{\widehat r_{0,\tau}(z)}.
\end{equation}
By Lemma~\ref{res:7-4},
\begin{equation}\label{e squ sig}
\min_X\kappa_{\tau,z}(\mu_\tau)
\le\sigma_\tau\le
\max_X\kappa_{\tau,z}(\mu_\tau).
\end{equation}
Combining (\ref{e asy r0tau}), (\ref{e kap es}) and (\ref{e squ sig}), we can get that:
\begin{equation}
    \sigma_\tau\longrightarrow 1.
\end{equation}
Every coordinate of $\sigma_\tau^{-1}\mu_\tau$ tends uniformly
to $+\infty$. Since $C_{Q_z}^{(1)}$ contains all sufficiently large
diagonal points and is upward stable,
$\sigma_\tau^{-1}\mu_\tau\in \mathcal{C}_{Q_z}^{(1)}$ for large $\tau$.
Proposition~\ref{prop:subsolution} makes $\underline\varphi$
a subsolution at $\sigma_\tau$. Theorem~\ref{thm:radial}
supplies a smooth elliptic solution with normalization
$c_\tau=\sigma_\tau$.

To identify this normalization, set
\[
I_k(\tau):=\int_X\omega_{0,\tau}^k\wedge\chi^{n-k},\qquad
A_\tau:=\int_X(b_0+\ell(\tau))\chi^n.
\]
Integrating $Q_{\tau,z}(c_\tau^{-1}\Lambda_{\varphi,\tau})=0$
gives $\mathcal R_\tau(c_\tau)=0$, where
\begin{equation}\label{eq:kahler-normalization-polynomial}
\mathcal R_\tau(c):=
I_d(\tau)-\sum_{k=1}^{d-2}b_kc^{d-k}I_k(\tau)-c^dA_\tau.
\end{equation}
Equation~\eqref{eq:kahler-path-compatibility} gives
$\mathcal R_\tau(1)=0$. Since
\[
I_k(\tau)=V\tau^k+O(\tau^{k-1}),\qquad
A_\tau=V\tau^d+O(\tau^{d-1}),
\]
we have
\begin{equation}\label{eq:kahler-normalization-asymptotic}
(V\tau^d)^{-1}\mathcal R_\tau(c)\longrightarrow1-c^d
\quad\text{in }C^1([1/2,2]).
\end{equation}
For large $\tau$, $\mathcal R_\tau$ is strictly decreasing on
$[1/2,2]$, which contains both $c_\tau$ and $1$.
Consequently $c_\tau=1$. Fix one such finite parameter $T$.
This constructs an unscaled initial solution without inferring
normalization from compatibility alone.

\emph{Step 4: continuation to the original background.}
For $0\le\tau\le T$, consider
\begin{equation}\label{eq:kahler-unscaled-path}
Q_{\tau,z}(\Lambda_{\varphi,\tau})=0,\qquad
\Lambda_{\varphi,\tau}\in \mathcal{C}_{Q_z}^{(1)},\qquad
\int_X\varphi\chi^n=0.
\end{equation}
Applying \eqref{eq:kahler-centered-zero-branch} to $Q_{\tau,z}$,
every solution has
$\Lambda_{\varphi,\tau}\in\partial \mathcal{C}_{Q_{\tau,z}}\subset\Gamma_d$.

For openness, adjoin an additive scalar $a$ to the equation.
Its linearization in $(\varphi,a)$ is $(w,b)\mapsto Lw-b$,
where $L$ is smooth and uniformly elliptic, with no zeroth-order
term. Lemma~\ref{res:4-2} and the implicit-function theorem apply.
Integration using \eqref{eq:kahler-path-compatibility} forces
$a=0$. The elliptic branch condition persists under sufficiently
small perturbations.

We verify the uniform estimate dependence needed for closedness.
Choose $\delta>0$ with $\mu-2\delta\one\in \mathcal{C}_{Q_z}^{(1)}$ everywhere.
Then $\mu_\tau-2\delta\one\in \mathcal{C}_{Q_z}^{(1)}$ for every $\tau\ge0$.
On the compact parameter interval $[0,T]$, the first derivatives
there have a common positive lower bound $\nu$, and all mixed
derivatives of orders $1,\ldots,d$ are positive. Multi-affinity gives
\begin{equation}\label{eq:kahler-uniform-slices}
Q_{\tau,z}(\mu_\tau-2\delta\one+y)
\ge Q_{\tau,z}(\mu_\tau-2\delta\one)+\nu\sum_i y_i
\qquad(y\ge0).
\end{equation}
The constant term on the right is bounded on $[0,T]\times X$,
so all buffered zero slices have one common radius $R$.

The backgrounds and reference forms have uniform smooth bounds.
The lower roots are fixed, and the smooth top root has a uniformly
positive first gap. In particular, the tuple
\[
\s(\tau,z):=(\widehat r_{0,\tau}(z),\widehat r_1,
\ldots,\widehat r_{d-1})
\]
has uniformly bounded spatial derivatives through order two.
Lemma~\ref{lem:C0-estimate} gives a uniform $C^0$ bound.
Proposition~\ref{res:5-1} implies that:
\[
\|\ddbar\varphi\|_{C^0}
\le C\bigl(1+\|\nabla\varphi\|_{C^0}^2\bigr)
\]
with $C$ independent of $\tau$ and the solution.

The blow-up proof of Lemma~\ref{res:5-5} is uniform as well. Thus, We obtain uniform gradient and complex
Hessian bounds. Apply Evans--Krylov and Schauder estimates to the
equivalent concave equation
$\kappa_{\tau,z}(\Lambda_{\varphi,\tau})=1$.
This gives uniform higher estimates.

The solvable set in $[0,T]$ is nonempty, open, and closed.
It therefore contains $0$. Since $\ell(0)=0$,
\eqref{eq:kahler-centered-polarization} translates this solution
back to the original equation $P_z(\Lambda_\varphi)=0$.
This proves (ii)$\Longrightarrow$(i); the reverse implication
is immediate.

\emph{Step 5: uniqueness and unshifted admissibility.}
Two centered solution forms lie in the convex spectral
$\Gamma_d$ domain and satisfy
$\kappa_{0,z}(\widehat\Lambda_\varphi)=1$.
Linearizing the radial operator along the segment between them,
as in the uniqueness proof of Theorem~\ref{thm:radial}, gives a
uniformly elliptic operator with no zeroth-order term annihilating
their potential difference. Its coefficients are uniformly
positive because the segment is compact inside the spectral
$\Gamma_d$ domain and the coordinate derivatives of
$\kappa_{0,z}$ are positive. The strong maximum principle proves
uniqueness modulo constants.
If $\gamma\ge0$, upward stability of $\Gamma_d$ applied to
\eqref{eq:kahler-canonical-admissibility} gives
$\Lambda_\varphi\in\Gamma_d$.
\end{proof}

\bibliographystyle{amsplain}
\bibliography{ref}

@book{Bha97,
  author    = {Bhatia, Rajendra},
  title     = {Matrix Analysis},
  series    = {Graduate Texts in Mathematics},
  volume    = {169},
  publisher = {Springer-Verlag},
  address   = {New York},
  year      = {1997},
  doi       = {10.1007/978-1-4612-0653-8},
  mrnumber = {1477662}
}

@misc{CNX26,
  author       = {Chen, Gao and Nie, Sijie and Xu, Yulun},
  title        = {A numerical criterion for complex {Hessian} type equations on projective manifolds},
  year         = {2026},
  howpublished = {Preprint, \href{https://arxiv.org/abs/2608.03815v2}{arXiv:2608.03815v2}},
  url          = {https://arxiv.org/abs/2608.03815v2}
}

@article{DK17,
  author  = {Dinew, S{\l}awomir and Ko{\l}odziej, S{\l}awomir},
  title   = {Liouville and {Calabi--Yau} type theorems for complex {Hessian} equations},
  journal = {Amer. J. Math.},
  volume  = {139},
  number  = {2},
  year    = {2017},
  pages   = {403--415},
  doi     = {10.1353/ajm.2017.0009},
  url     = {https://arxiv.org/abs/1203.3995},
  mrnumber = {3636634}
}

@article{JPDemaillyPaun,
  author = {Demailly, Jean-Pierre and Paun, Mihai},
  title = {Numerical characterization of the {K}\"{a}hler cone of a compact {K}\"{a}hler manifold},
  journal = {Ann. of Math. (2)},
  fjournal = {Annals of Mathematics. Second Series},
  volume = {159},
  year = {2004},
  number = {3},
  pages = {1247--1274},
  issn = {0003-486X,1939-8980},
  mrclass = {32J27 (32Q15)},
  mrnumber = {2113021},
  mrreviewer = {Philippe P. Eyssidieux},
  doi = {10.4007/annals.2004.159.1247},
  url = {https://doi.org/10.4007/annals.2004.159.1247},
}

@article{GaoChen2021,
  author = {Chen, Gao},
  title = {The {J}-equation and the supercritical deformed {H}ermitian-{Y}ang-{M}ills equation},
  journal = {Invent. Math.},
  fjournal = {Inventiones Mathematicae},
  volume = {225},
  year = {2021},
  number = {2},
  pages = {529--602},
  issn = {0020-9910,1432-1297},
  mrclass = {53C55},
  mrnumber = {4285141},
  mrreviewer = {Adam Jacob},
  doi = {10.1007/s00222-021-01035-3},
  url = {https://doi.org/10.1007/s00222-021-01035-3},
}

@article{Datar2020ANC,
  author = {Datar, Ved V. and Pingali, Vamsi Pritham},
  title = {A numerical criterion for generalised {M}onge-{A}mp\`ere equations on projective manifolds},
  journal = {Geom. Funct. Anal.},
  fjournal = {Geometric and Functional Analysis},
  volume = {31},
  year = {2021},
  number = {4},
  pages = {767--814},
  issn = {1016-443X,1420-8970},
  mrclass = {53C55 (32W20 35K96)},
  mrnumber = {4317503},
  mrreviewer = {Anna M. Fino},
  doi = {10.1007/s00039-021-00577-1},
  url = {https://doi.org/10.1007/s00039-021-00577-1},
}

@article{ChuLeeTakhashi,
  author = {Chu, Jianchun and Lee, Man-Chun and Takahashi, Ryosuke},
  title = {A {N}akai--{M}oishezon type criterion for supercritical deformed {H}ermitian--{Y}ang--{M}ills equation},
  journal = {J. Differential Geom.},
  fjournal = {Journal of Differential Geometry},
  volume = {126},
  year = {2024},
  number = {2},
  pages = {583--632},
  issn = {0022-040X,1945-743X},
  mrclass = {53C55 (35A01 35J60)},
  mrnumber = {4728704},
  mrreviewer = {Kotaro Kawai},
  doi = {10.4310/jdg/1712344220},
  url = {https://doi.org/10.4310/jdg/1712344220},
}

@article{CollinsGabor,
  author = {Collins, Tristan C. and Sz\'{e}kelyhidi, G\'{a}bor},
  title = {Convergence of the {$J$}-flow on toric manifolds},
  journal = {J. Differential Geom.},
  fjournal = {Journal of Differential Geometry},
  volume = {107},
  year = {2017},
  number = {1},
  pages = {47--81},
  issn = {0022-040X,1945-743X},
  mrclass = {53C44 (53C55)},
  mrnumber = {3698234},
  mrreviewer = {Mijia Lai},
  doi = {10.4310/jdg/1505268029},
  url = {https://doi.org/10.4310/jdg/1505268029},
}

@misc{FM26a,
  author       = {Fang, Hao and Ma, Biao},
  title        = {{G{\aa}rding} polynomials},
  year         = {2026},
  howpublished = {Preprint, \href{https://arxiv.org/abs/2604.27755v2}{arXiv:2604.27755v2}},
  url          = {https://arxiv.org/abs/2604.27755v2}
}

@misc{FM26b,
  author       = {Fang, Hao and Ma, Biao},
  title        = {Ideal {G{\aa}rding} polynomials},
  year         = {2026},
  howpublished = {Preprint, \href{https://arxiv.org/abs/2607.16832v1}{arXiv:2607.16832v1}},
  url          = {https://arxiv.org/abs/2607.16832v1}
}

@misc{fu2026nakaimoishezontypecriteriahessian,
      title={Nakai-Moishezon type criteria for Hessian type equations on compact K\"ahler manifolds}, 
      author={Jixiang Fu and Xuan Li and Dekai Zhang and Ziyi Zhang},
      year={2026},
      eprint={2609.40238},
      archivePrefix={arXiv},
      primaryClass={math.DG},
      url={https://arxiv.org/abs/2609.40238}, 
}

@misc{GS24,
  author       = {Guo, Bin and Song, Jian},
  title        = {Sup-slopes and sub-solutions for fully nonlinear elliptic equations},
  year         = {2024},
  howpublished = {Preprint, \href{https://arxiv.org/abs/2405.03074v1}{arXiv:2405.03074v1}},
  url          = {https://arxiv.org/abs/2405.03074v1}
}

@article{song2020nakaimoishezoncriterionscomplexhessian,
  author = {{Song}, Jian},
  title = {Nakai--Moishezon criterions for complex Hessian equations},
  journal = {arXiv e-prints},
  year = {2020},
  month = dec,
  eid = {arXiv:2012.07956},
  pages = {arXiv:2012.07956},
  doi = {10.48550/arXiv.2012.07956},
  archiveprefix = {arXiv},
  eprint = {2012.07956},
  primaryclass = {math.DG},
  adsurl = {https://ui.adsabs.harvard.edu/abs/2020arXiv201207956S},
}

@article{HMW10,
  author  = {Hou, Zuoliang and Ma, Xi-Nan and Wu, Damin},
  title   = {A second order estimate for complex {Hessian} equations on a compact {K{\"a}hler} manifold},
  journal = {Math. Res. Lett.},
  volume  = {17},
  number  = {3},
  year    = {2010},
  pages   = {547--561},
  doi     = {10.4310/MRL.2010.v17.n3.a12},
  mrnumber = {2653687}
}

@article{Lin23,
  author  = {Lin, Chao-Ming},
  title   = {On the convexity of general inverse {$\sigma_k$} equations},
  journal = {J. Funct. Anal.},
  volume  = {285},
  number  = {7},
  year    = {2023},
  pages   = {Paper No. 110038},
  note    = {59 pp.},
  doi     = {10.1016/j.jfa.2023.110038},
  url     = {https://arxiv.org/abs/2209.11370}
}

@book{MOA11,
  author    = {Marshall, Albert W. and Olkin, Ingram and Arnold, Barry C.},
  title     = {Inequalities: Theory of Majorization and Its Applications},
  edition   = {Second},
  series    = {Springer Series in Statistics},
  publisher = {Springer},
  address   = {New York},
  year      = {2011},
  doi       = {10.1007/978-0-387-68276-1},
  mrnumber = {2759813}
}

@article{Collins201511FW,
  author = {Collins, Tristan C. and Jacob, Adam and Yau, Shing-Tung},
  title = {{$(1,1)$} forms with specified {L}agrangian phase: a priori estimates and algebraic obstructions},
  journal = {Camb. J. Math.},
  fjournal = {Cambridge Journal of Mathematics},
  volume = {8},
  year = {2020},
  number = {2},
  pages = {407--452},
  issn = {2168-0930,2168-0949},
  mrclass = {53C07 (14F08 14J32 53C55)},
  mrnumber = {4091029},
  mrreviewer = {Weisong Dong},
  doi = {10.4310/cjm.2020.v8.n2.a4},
  url = {https://doi.org/10.4310/cjm.2020.v8.n2.a4},
}

@article{Sze18,
  author  = {Sz{\'e}kelyhidi, G{\'a}bor},
  title   = {Fully non-linear elliptic equations on compact {Hermitian} manifolds},
  journal = {J. Differential Geom.},
  volume  = {109},
  number  = {2},
  year    = {2018},
  pages   = {337--378},
  doi     = {10.4310/JDG/1527040875},
  note    = {Numbered citations refer to \href{https://arxiv.org/abs/1501.02762v3}{arXiv:1501.02762v3}},
  url     = {https://arxiv.org/abs/1501.02762v3},
  mrnumber = {3807322}
}

@article{TWWY15,
  author  = {Tosatti, Valentino and Wang, Yu and Weinkove, Ben and Yang, Xiaokui},
  title   = {{$C^{2,\alpha}$} estimates for nonlinear elliptic equations in complex and almost complex geometry},
  journal = {Calc. Var. Partial Differential Equations},
  volume  = {54},
  number  = {1},
  year    = {2015},
  pages   = {431--453},
  doi     = {10.1007/s00526-014-0791-0},
  url     = {https://arxiv.org/abs/1402.0554},
  mrnumber = {3385166}
}

@article{CNS85,
  author   = {Caffarelli, Luis and Nirenberg, Louis and Spruck, Joel},
  title    = {The {Dirichlet} problem for nonlinear second-order elliptic equations.
              {III}. {Functions} of the eigenvalues of the {Hessian}},
  journal  = {Acta Math.},
  volume   = {155},
  number   = {3-4},
  year     = {1985},
  pages    = {261--301},
  doi      = {10.1007/BF02392544},
  mrnumber = {806416}
}

@article{Song2004OnTC,
  author = {Song, Jian and Weinkove, Ben},
  title = {On the convergence and singularities of the {$J$}-flow with applications to the {M}abuchi energy},
  journal = {Comm. Pure Appl. Math.},
  fjournal = {Communications on Pure and Applied Mathematics},
  volume = {61},
  year = {2008},
  number = {2},
  pages = {210--229},
  issn = {0010-3640,1097-0312},
  mrclass = {32Q20},
  mrnumber = {2368374},
  mrreviewer = {Julius Ross},
  doi = {10.1002/cpa.20182},
  url = {https://doi.org/10.1002/cpa.20182},
}

@article{TW10,
  author   = {Tosatti, Valentino and Weinkove, Ben},
  title    = {The complex {Monge--Amp{\`e}re} equation on compact {Hermitian} manifolds},
  journal  = {J. Amer. Math. Soc.},
  volume   = {23},
  number   = {4},
  year     = {2010},
  pages    = {1187--1195},
  doi      = {10.1090/S0894-0347-2010-00673-X},
  mrnumber = {2669712}
}

@article{Yau78,
  author   = {Yau, Shing-Tung},
  title    = {On the Ricci curvature of a compact {K\"ahler} manifold and the complex {Monge--Amp\`ere} equation, I},
  journal  = {Comm. Pure Appl. Math.},
  volume   = {31},
  number   = {3},
  year     = {1978},
  pages    = {339--411},
  doi      = {10.1002/cpa.3160310304},
  mrnumber = {0480350}
}

@article{fang2010classfullynonlinearflow,
  author = {Fang, Hao and Lai, Mijia and Ma, Xinan},
  title = {On a class of fully nonlinear flows in {K}\"{a}hler geometry},
  journal = {J. Reine Angew. Math.},
  fjournal = {Journal f\"{u}r die Reine und Angewandte Mathematik. [Crelle's Journal]},
  volume = {653},
  year = {2011},
  pages = {189--220},
  issn = {0075-4102,1435-5345},
  mrclass = {53C44 (32Q15 32W20 53C55)},
  mrnumber = {2794631},
  mrreviewer = {Yanir A. Rubinstein},
  doi = {10.1515/CRELLE.2011.027},
  url = {https://doi.org/10.1515/CRELLE.2011.027},
}

@article{Fang2023OnAF,
  author = {Fang, Hao and Ma, Biao},
  title = {On a fully nonlinear elliptic equation with differential forms},
  journal = {Adv. Math.},
  fjournal = {Advances in Mathematics},
  volume = {454},
  year = {2024},
  pages = {Paper No. 109867, 89},
  issn = {0001-8708,1090-2082},
  mrclass = {53C55},
  mrnumber = {4781476},
  mrreviewer = {Martin de Borbon},
  doi = {10.1016/j.aim.2024.109867},
  url = {https://doi.org/10.1016/j.aim.2024.109867},
}

@ARTICLE{2026arXiv260711296F,
       author = {{Fu}, Jixiang and {Zhang}, Dekai and {Zhang}, Ziyi},
        title = "{A Numerical Criterion for the 2-Hessian Equation on Compact K{\"a}hler Manifolds}",
      journal = {arXiv e-prints},
         year = 2026,
        month = jul,
          eid = {arXiv:2607.11296},
        pages = {arXiv:2607.11296},
          doi = {10.48550/arXiv.2607.11296},
archivePrefix = {arXiv},
       eprint = {2607.11296},
 primaryClass = {math.DG},
       adsurl = {https://ui.adsabs.harvard.edu/abs/2026arXiv260711296F}
}

@article{Donaldson99,
  author = {Donaldson, Simon K.},
  title = {Moment maps and diffeomorphisms},
  journal = {Asian J. Math.},
  fjournal = {Asian Journal of Mathematics},
  volume = {3},
  number = {1},
  year = {1999},
  pages = {1--16},
  doi = {10.4310/ajm.1999.v3.n1.a1},
  url = {https://doi.org/10.4310/ajm.1999.v3.n1.a1},
}

@article{Zhang17,
  author = {Zhang, Dekai},
  title = {Hessian equations on closed {Hermitian} manifolds},
  journal = {Pacific J. Math.},
  fjournal = {Pacific Journal of Mathematics},
  volume = {291},
  number = {2},
  year = {2017},
  pages = {485--510},
  doi = {10.2140/pjm.2017.291.485},
  url = {https://doi.org/10.2140/pjm.2017.291.485},
}

@article{Lin23solv,
  author = {Lin, Chao-Ming},
  title = {On the Solvability of General Inverse {$\sigma_k$} Equations},
  journal = {arXiv e-prints},
  year = {2023},
  eid = {arXiv:2310.05339},
  pages = {arXiv:2310.05339},
  doi = {10.48550/arXiv.2310.05339},
  archivePrefix = {arXiv},
  eprint = {2310.05339},
  primaryClass = {math.DG},
  url = {https://arxiv.org/abs/2310.05339},
}

@article{Lin24corr,
  author = {Lin, Chao-Ming},
  title = {Corrigendum to ``On the convexity of general inverse
           {$\sigma_k$} equations'' [J. Funct. Anal. 285(7)
           (2023) 110038]},
  journal = {J. Funct. Anal.},
  fjournal = {Journal of Functional Analysis},
  volume = {286},
  number = {11},
  year = {2024},
  pages = {Paper No. 110398},
  doi = {10.1016/j.jfa.2024.110398},
  url = {https://doi.org/10.1016/j.jfa.2024.110398},
}

@misc{fang2026spectralgaardingelliptic,
      title={Spectral G\aa rding Elliptic Equations}, 
      author={Hao Fang and Biao Ma},
      year={2026},
      eprint={2610.15434},
      archivePrefix={arXiv},
      primaryClass={math.AP},
      url={https://arxiv.org/abs/2610.15434}, 
}

\end{document}